\documentclass[11pt]{article}
\usepackage{amsfonts,amsmath,latexsym,verbatim,amscd,amsthm,graphics}
\newtheorem{theorem}{Theorem}[section]
\newtheorem{proposition}{Proposition}[section]
\newtheorem{corollary}{Corollary}[section]
\newtheorem{lemma}{Lemma}[section]
\newtheorem{definition}{Definition}[section]
\newtheorem{remark}{Remark}[section]
\numberwithin{equation}{section}
\numberwithin{lemma}{section}
\numberwithin{theorem}{section}
\numberwithin{definition}{section}
\numberwithin{proposition}{section}

\newcommand{\RR}{\mathbb{R}}

\newcommand{\HH}{\mathbb{H}}
\newcommand{\e}{\epsilon}
\newcommand{\del}{\partial}

\newcommand{\ol}{\overline}
\newcommand{\g}{{\bf g}}

\newcommand{\calC}{{\mathcal C}}
\newcommand{\calD}{{\mathcal D}}
\newcommand{\calE}{{\mathcal E}}

\newcommand{\calM}{{\mathcal M}}

\newcommand{\calO}{{\mathcal O}}
\newcommand{\calR}{{\mathcal R}}
\newcommand{\calS}{{\mathcal S}}

\newcommand{\calU}{{\mathcal U}}

\newcommand{\olg}{\overline{g}}

\newcommand{\rena}{\mathrm{RenA}}

\newcommand{\LipRad}{\mathrm{LipRad}}

\newcommand{\Aring}{\mathring{A}}
\newcommand{\Abring}{\mathring{\bar{A}}}
\newcommand{\Abar}{\bar{A}}

\newcommand{\barg}{\bar{g}}
\newcommand{\EE}{{\mathbb{E}^2}}

\newcounter{mnotecount}[section]

\title{Complete Willmore surfaces in $\HH^3$ with bounded energy:\\ 
boundary regularity and bubbling.}
\author{Spyros Alexakis \thanks{Email: alexakis@math.toronto.edu. 
}\\ University of Toronto\and Rafe Mazzeo \thanks{Email: mazzeo@math.stanford.edu} \\ Stanford University}
\date{}

\begin{document}
\maketitle
\abstract{We study various aspects related to boundary regularity of complete properly embedded Willmore surfaces in $\HH^3$,
particularly those related to assumptions on boundedness or smallness of a certain weighted version of the 
Willmore energy.  We prove, in particular, that small energy 
controls $\calC^1$ boundary regularity.  We examine the possible lack of $\calC^1$ convergence for sequences of surfaces
with bounded Willmore energy and find that the mechanism responsible for this is a bubbling phenomenon, where 
energy escapes to infinity.} 

\section{Introduction}
In our previous paper \cite{AM} we studied the renormalized area, $\rena(Y)$, as a functional on the space of all 
properly embedded minimal surfaces $Y$ in $\HH^3$ with a sufficiently smooth boundary curve at infinity. 
Area or volume renormalization of a properly embedded {\it minimal} submanifold of arbitrary dimension or codimension in
hyperbolic space was introduced by Graham and Witten \cite{GW}; the renormalization is accomplished by 
an Hadamard regularization of the asymptotic expansion of areas (or volumes) of a family of compact truncations
of the submanifold. The renormalized area of such a minimal surface in $\HH^3$ turns out to be a classical quantity. The first
result in \cite{AM} is that 
\begin{equation}
\rena(Y) = -2\pi \chi(Y) - \frac12 \int_Y |\Aring|^2\, d\mu,
\label{defrena}
\end{equation}
where $\chi(Y)$ is the Euler characteristic and $\Aring$ the trace-free second fundamental form of $Y$.  Since $Y$ is minimal, 
$\Aring$ equals the full second fundamental form, and so $\int_Y |\Aring|^2\, d\mu$ is the same as the total curvature 
$\int_Y |A|^2\, d\mu$ of the surface $Y$. In other words, up to the topological contribution, $\rena(Y)$ is essentially
the same as the Willmore energy $\calE(Y):=\int_Y|A|^2\, d\mu$ of $Y$.

We can also relate $\calE(Y)$ to the Willmore energy of $\overline{Y}$, regarded as a surface with boundary in the upper 
half-space $\RR^3_+$ with the Euclidean metric. Indeed, the density $|\Aring|^2\, d\mu$ is well-known to be invariant 
with respect to conformal changes of the ambient space. If $\overline{Y}$ is $\calC^2$ up to its boundary, then it 
meets $\del \RR^3_+$ orthogonally and one may form a closed surface 
by doubling $\overline{Y}$ across this boundary. Decorating all quantities computed with respect to the 
Euclidean metric with bars, then a short calculation using the Gauss-Bonnet theorem yields the relationship
\begin{equation}
\label{gauss-bonnet.appln} 
\int_Y|\Abar|^2 \, d\bar{\mu}=\int_Y |\Abring|^2\, d\bar{\mu} +4\pi\chi(Y)=\int_Y |\Aring|^2\, d\mu +4\pi\chi(Y).
\end{equation}

Based on these observations, it is natural to regard this Willmore energy $\calE$ as a functional on the space of complete minimal surfaces 
in $\HH^3$ (with $\calC^2$ asymptotic boundary); henceforth we refer to this simply as the energy functional. Although minimal 
surfaces are always critical points of $\calE$, it has many other critical points as well. These are the so-called Willmore 
surfaces. The Willmore equation, i.e.\ the Euler-Lagrange equation for $\calE$, is conformally invariant, hence
we may simply talk about Willmore surfaces for either the hyperbolic or the Euclidean metric on the upper half-space.
In any case, the objects of study in this paper are the Willmore surfaces in the upper half-space which meet
$\del \HH^3$ orthogonally, and the energy $\calE$ as a functional on this class of surfaces. 
It is important to note that the Willmore {\it energy} is not conformally invariant. As we prove below,
the finiteness of the (hyperbolic) Willmore energy of a Willmore surface $Y$ necessitates that $Y \perp \del \HH^3$,
so this condition can be dropped. Although we primarily work using
the ambient hyperbolic metric, it is useful to recall the Euclidean perspective occasionally too. 

Our aim is to study sequences $Y_j$ of Willmore surfaces with fixed genus and number of ends, and in particular, to examine 
how the boundedness of $\calE(Y_j)$ regulates convergence of the $Y_j$ at their boundaries. It turns out that the natural energy
functional $\calE$ does not seem to control boundary regularity in a sufficiently strong manner, but that one does obtain
such control using an appropriately weighted version $\calE_p$ of this energy. This new functional is defined by
inserting an extra weight factor in the integral of $|\Aring|^2$:
\[
\calE_p(Y):=\int_Y|\Aring|^2 f^{2p}\, d\mu,
\]
where $p > 1$ is fixed and $f:Y\to \RR_+$ is the intrinsic distance in $Y$ to a given finite collection of points in $Y$, 
called {\it poles}. These poles are in the interior of $Y$ so that near $\del_\infty\HH^3=\{x=0\}$ (in the upper half-space 
model), $f\sim |\log x|$. For brevity we refer to ${\cal E}_p(Y)$ as the {\it weighted energy} of $Y$.

We shall study the following problem: If $Y_j$ is a sequence of Willmore surfaces 
with $\calE_p(Y_j)\le C <\infty$, then does some subsequence of the $Y_j$ converge 
in $\calC^1$ up to the boundary?  In fact, we show that $\calC^1$ convergence may fail at a finite set of points
at the boundary, but we are able to understand this phenomenon via the loss of energy in the limit. Since convergence 
of Willmore surfaces in any compact set of $\HH^3$ is well understood, we focus almost entirely on the behavior of these 
surfaces near and at their asymptotic boundaries. 

Before stating our results, we put this into a broader context.  The study of failure of compactness for variational problems 
goes back at least to \cite{SU} and has now been explored in a wide variety of settings; we refer to \cite{rivICM} for 
a good overview of results and methods. Particularly relevant to our problem are the many deep advances in
understanding the analytic aspects of the Willmore functional; we refer in particular to the fundamental paper of L.\ Simon
\cite{Si}, the more recent work by Kuwert and Sch\"atzle \cite{KuS} and the powerful new approach developed by
Rivi\`ere \cite{riv1}, see also \cite{MR}. However, none of these papers (on Willmore or otherwise), to our knowledge, deal 
with this loss of compactness due to bubbling at the boundary. Often this failure of compactness at the boundary
is excluded by imposing apriori bounds on boundary regularity. Our particular geometric problem presents a natural situation 
where it is unnatural to impose such boundary control, and where this bubbling phenomenon occurs. We note,
however, that regularity at a free boundary for submanifolds with prescribed mean curvature has been studied in \cite{KS}. 

The second context in which to view our work is slightly more tenuous. To explain it we first recall the computation from 
\cite{AM} which gives the first variation of $\calE$ at a minimal surface $Y$. If $\gamma = \del_\infty Y$ is the boundary curve 
at infinity, then there is function $u_3$ associated to $Y$ such that
\begin{equation}
\left.D\, \calE\right|_Y(\psi) = 6 \int_\gamma u_3 \psi_0 \, ds.
\end{equation} 
Here $\psi$ is a Jacobi field along $Y$, i.e.\ an infinitesimal variation of $Y$ amongst minimal surfaces and $\psi_0$ 
its boundary value at $\gamma$, and $s$ is the arclength parameter along $\gamma$.  A case was made in \cite{AM}
that the pair $(\gamma, u_3)$ should be regarded as the Cauchy data of $Y$.  It follows from the basic regularity theory for
such surfaces, due to Tonegawa \cite{To}, that if the `Dirichlet data' $\gamma$ is $\calC^\infty$, then
$\overline{Y}$ is $\calC^\infty$ up to the boundary. Based on classical elliptic theory, one might also expect that
control on the Neumann data, $u_3$, should also yield regularity of $Y$ up to the boundary. In particular, 
if $Y_j$ is a Palais-Smale sequence for $\calE$ (or $\calE_p$), then the functions $u_3^{(j)}$ converge to zero in some 
weak sense, and the question then becomes whether quantitative measures of smallness on these functions yield greater 
control on the boundary curves $\gamma_j$.  We do not emphasize this point of view, however, since it is difficult to make precise. 

\medskip

{\bf Results:}  We first prove an $\e$-regularity result: if the weighted energy of a Willmore surface in a Euclidean half-ball in the upper 
half-space model around some point $P\in\del_\infty Y$ is small, then  the $\calC^1$ norm of the surface is controlled 
uniformly up to the boundary. This has the following analytic content: regarding $Y$ as a horizontal 
graph over a vertical half-plane, then finiteness of the weighted energy is slightly weaker than bounding the graph function 
in $W^{2,2}$, with a logarithmic weight. Hence $\calC^1$ regularity shows 
that this graph function exhibits {\it better} behavior near the boundary than 
would follow from the Sobolev embedding theorem. 
This $\calC^1$ regularity is nearly optimal. Indeed, the energy ${\calE}_p$ is dilation-invariant, since dilations are hyperbolic isometries,
but if we take a blow-down 
limit of a given surface, then the $\calC^{1,\alpha}$ norm of the boundary curve diverges, so we could not expect that
norm to be controlled by the weighted energy.  It is not clear how to characterize the optimal regularity associated 
with finiteness of weighted energy, nor is is it obvious whether there is an optimal weight function that guarantees $\calC^1$ regularity.

One application of this first result is that if $Y_j$ is a sequence of Willmore surfaces with $\calE_p$ bounded, and with well-separated
boundary components, then some subsequence converges  to a Willmore surface $Y_*$, the boundary at infinity of which is a priori 
Lipschitz except at a finite number of bad points. We then show that except possibly at these exceptional points, the limit curve is 
$\calC^1$. This is a gain of regularity compared to Sobolev embedding. 
We note that the convergence of $\gamma_j = \del_\infty Y_j$ to $\gamma_* = \del_\infty Y_*$ need not be $\calC^1$;
in fact we construct counterexamples to this at the end of this paper: Using fairly simple gluing arguments, we obtain
a sequence $Y_j$ with energy $\calE_p(Y_j) \leq C$ which converges to a totally geodesic hemisphere, but where 
the convergence is not $\calC^1$ at a finite number of boundary points.  At each of those points, one sees a 
sequence of increasingly strong blow-downs of a fixed Willmore surface, which carries a fixed positive amount of 
 energy, shrink to a point;  we regard this as a type of bubbling. However, unlike the various `interior' bubbling phenomena 
mentioned earlier  which only occur when the energy is above  a certain threshold, 
in this setting  {\it arbitrarily small} amounts of energy can disappear in these limits. 

Our final result is that the phenomenon exhibited by these examples above is the only mechanism 
through which the convergence $Y_j\to Y_*$ can fail to be $\calC^1$ near the boundary, at least in regions 
of small energy.  In such regions we show that if $P_j\in \gamma_j$, $P_j \to P_* \in \gamma_*$, but the tangent lines
$T_{P_j} \gamma_j$ fail to converge to $T_{P_*} \gamma_*$, then there exist a sequence of hyperbolic
isometries $\varphi_j$ which dilate away from $P_j$ and are such that $\varphi_j(Y_j) \to \tilde{Y}_*$, where 
$\calE(\tilde{Y}_*) > 0$.  Finally, we show that such a bubble of energy (which is already receding to infinity before
applying the dilations $\varphi_j$) carries with it one of the poles used to define the weight function $f$. 
The investigation of regularity gain and bubbling in regions of large energy presents various technical  difficulties 
(some of which are already apparent in \cite{Li}) which are beyond the scope of this paper. 
We intend to return to this in the future. 

We now provide a brief outline of some of the key ideas and arguments in this paper. The preamble to each section 
contains more extensive discussion of the main idea in that section. 

\medskip

{\bf Outline:}   The argument commences in \S 2, where we
prove two ``soft'' results about boundary regularity for Willmore surfaces with finite energy. Together, these 
show that any such surface must meet $\del_\infty \HH^3$ orthogonally and have a good local graphical
representation over a vertical plane provided the boundary curve has a corresponding graphical representation
over a line. This relies only on interior regularity results for Willmore surfaces and simple Morse-theoretic arguments.

\medskip

\noindent {\it $\e$-regularity:} In our first ``hard'' result, we prove that for  (local) Willmore surfaces with sufficiently
small weighted energy,
 one obtains $\calC^1$ control on the boundary curve.  Indeed, we argue that if this were to
fail, then one could construct a sequence of Willmore surfaces, the energies of which vanish in the limit, but such
that there is a jump in the tangent lines in the limit. To reach a contradiction, we wish to relate the slope of the tangent line
at the boundary to information on a parallel curve in the interior of the surface and then use the known $\calC^\infty$
convergence in the interior. 

The relationship between derivative information in the interior and at the boundary, i.e.\ the difference
between the `horizontal' derivatives at height $0$ and $1$, say, is given by integrating the mixed second 
derivative of the graph function along a vertical line and showing that this is controlled by the energy.
To do this we must use a choice of `gauge', which is a special isothermal coordinate system for
which we have explicit pointwise control of the conformal factor. Using some deep results in harmonic 
analysis, such coordinate systems have been obtained for related problems, e.g.\ for embedded spheres
by De Lellis and M\"uller \cite{DLM}, following an earlier and very influential paper by M\"uller and Sverak \cite{MS}, 
see also H\'elein \cite{H};  We must modify those arguments to our setting, which requires a `preparation' of our 
surface in a couple of ways. We first locally graft  our surfaces into a round sphere so that the resulting non-round sphere 
 has two reflection symmetries.  We then apply an appropriate 
M\"obius transformation to normalize the positioning of this surface so as to be in a position to apply
the results in \cite{DLM}, \cite{MS}. Throughout this whole procedure we must be careful that none
of these alterations change the fact that there is a jump in first derivatives at the 
origin. Note that this requires the finiteness of the unweighted energy only.
This boundedness of the conformal factor, along with the boundedness of the weighted energy 
allow us to obtain pointwise control of  the mixed component of the 
2nd derivative of the graph function and bounding its line integral.
This pointwise bound relies on a realization of Willmore surfaces as harmonic maps into the deSitter space; 
a mean value inequality for this map yields a bound on a specific component of the second fundamental form,  
which in turn implies our desired pointwise bounds. It is at this point that the finiteness of the weighted 
Willmore energy is essential. 
In this argument there is a second line integral which it is necessary to control in terms of 
the energy of $Y_j$ in a half-ball. This second line integral (which can be controlled by the regular rather than the weighted energy) 
 plays a crucial role in the later analysis of bubbling.

These arguments occupy \S 3-6. In \S 7 we use the techniques developed up to that point to derive 
the regularity gain for the limit surface $Y_*$ in regions of small energy. 

\medskip

{\it Bubbling:}  Section 8 contains the argument that if the convergence $Y_j\to Y_*$ is not $\calC^1$ at 
some sequence of boundary points $P_j\to P_*\in \del_\infty Y_*$, then we can perform a sequence of 
blowups near those boundary points to produce a sequence of Willmore surfaces $\tilde{Y}_j$ 
which converge to a limit surface $\tilde{Y}_*$ of non-zero energy; prior to the blow-up the surfaces $\tilde{Y}_j$ are
disappearing in the limit towards $P_*$. 
In other words, the $\calC^1$ loss of compactness is due to portions of $Y_j$ with fixed (but arbitrary)
nonzero energy disappearing at infinity.   Unlike similar arguments for bubbling in the interior, 
since our surfaces have infinite area, it is not initially clear that we can find points $Y_j \ni Q_j \to \del_\infty\HH^3$ 
on which $|\Aring|_g$ is bounded below; the rescalings we wish to perform should be centered at such points.
Their existence is proved indirectly, by arguing that it is impossible for all possible blowups near the points 
$P_j\in Y_j$ to converge to surfaces of zero energy. This argument makes essential use of the second line integral 
mentioned above. The key point is to show that this line integral can be controlled by the energy of $Y_j$
in a conical region emanating from (rather than a half-ball containing) $P_j$. 

\medskip
{\bf Further questions:}  There are several questions and problems which are closely related to the themes in
this paper and which seem particularly interesting. We hope to return to some of these soon.

Despite the fact that the problems which led us to the current investigations involve Willmore surfaces of finite weighted 
energy in $\mathbb{H}^3$, one could equaly study Willmore surfaces in the Euclidean ball, 
with boundaries lying on the boundary $\mathbb{S}^2$. In fact, the present work makes clear that
only a weighted version of the traceless part 
$\int_Y|\Aring|^2f^{2p}d\mu$ of the total curvature is needed for our results; 
in view of the conformal invariance of the form $|\Aring|^2d\mu$, 
this suggests that the results here may also hold in a Euclidean ball, assuming an upper bound  on the weighted energy
$\int_Y|\Abring|^2f^{2p}d\ol{\mu}$ and imposing bounds on the angle of intersection between $Y$ and $\mathbb{S}^2 = 
\del B^3$. Indeed, many of the methods developed here transfer to that more general setting with no difficulty. 

Another question, which was a motivation for this work but not studied explicitly here,
concerns the analysis of sequences of Willmore surfaces $Y_j$ which are Palais-Smale for the functional $\calE$. 
Recall that this means that $\calE(Y_j)$ tends to a critical value and $\left. D\calE\right|_{Y_j}$ converges to $0$. 
The goal would be to find critical points for $\calE$. Our results show that critical sequences 
may converge to surfaces with strictly lower genus, and this convergence often occurs only in a weak norm 
at the boundary, but it may still be possible to produce $\calE$-critical surfaces this way.

Finally, one other set of problems we wish to mention involve an analogous though more complicated
problem of studying sequences of Poincar\'e-Einstein metrics in four dimensions. Recall that $(M,g)$ is said
to be Poincar\'e-Einstein if $M$ is a compact manifold with boundary, and $g$ is conformally compact (hence
is complete on the interior of $M$) and Einstein, see \cite{MP, An2} for more details and further references. 
These objects can be studied in any dimension, but it is known that dimension $4$ is critical in the same way that 
dimension $2$ is critical for Willmore surfaces. This is reflected in two formul\ae\ due to Anderson
\cite{An1}: the first is an explicit local integral expression for the renormalized volume of a 
four-dimensional Poincar\'e-Einstein space as the sum of an Euler characteristic and the squared 
$L^2$-norm of the Weyl curvature, while the second describes the differential of renormalized volume with respect
to Poincar\'e-Einstein deformations. These are entirely analogous to (and indeed were the motivatations for) 
the corresponding formul\ae\ here. It is therefore reasonable to ask whether results like the ones here can be proved 
in that Poincar\'e-Einstein setting.  Slightly more generally, reflecting the passage from minimal to Willmore,
one should study these questions in the setting of Bach-flat metrics. More specifically, suppose that $(M^4,g_j)$ is a sequence 
of Poincar\'e-Einstein (or Bach-flat) metrics such that $\int |W_j|^2\, dV_{g_j} \leq C < \infty$, where $W_j$ is the Weyl
tensor of $g_j$. The specific issue is to determine how this uniform energy bound (or some suitably weighted version of
such a bound) controls the regularity of the sequence of conformal infinities of $g_j$. This is related to the questions 
studied by Anderson \cite{An2} and more recently by Chang-Qing-Yang \cite{CQY}. 

\medskip

\subsection{Notation and terminology} 
Almost all of the results below are local, so we always work in the upper half-space model $\RR^3_+$ of $\HH^3$, 
with vertical (height) coordinate $x$, and with linear coordinates $(y, z)$ on $\RR^2 = \{x=0\}$.  

All of the surfaces studied here are assumed to be smooth and Willmore (or minimal, if noted explicitly). We always assume
that any such $Y$ is connected and has closure $\overline{Y} \subset \overline{\HH^3}$, a compact surface with 
boundary curve $\gamma = \del_\infty Y \subset \RR^2$ which is embedded and closed, but possibly disconnected. 
We assume that $\overline{Y}$ is at least $\calC^2$ unless explicitly stated otherwise. 
Since $\HH^3$ has many isometries, including dilation and horizontal ($\RR^2$) translation, it is convenient 
to fix a normalization and scale. We say that $Y$ is normalized if the length of its boundary curve (measured with respect 
to the Euclidean metric on $\RR^2$) satisfies $|\gamma| = 100\pi$ and if the center of mass of $\gamma$ in $\RR^2$ is $0$. 
The class of all Willmore surfaces with $k$ ends and genus $g$, normalized in this way, and which meet $\del \RR^3_+$ 
orthogonally, is denoted $\calM_{k,g}$, and $\calM=\bigcup_{k,g}\calM_{k,g}$. For each $Y\in\calM$, $\overline{Y}$ is the closure 
of $Y$ in $\RR_+^3$. 

Many of the arguments below use the interplay between the metrics $g$ and 
$\bar{g}$ on a surface $Y$ induced from the ambient hyperbolic and Euclidean metrics, respectively. We denote by $A$ 
and $\Abar$, and $\Aring$ and $\Abring$ the corresponding second fundamental form and trace-free second fundamental 
forms of $Y$ and by $d\mu, d\overline{\mu}$ the area elements. As noted earlier, 
\begin{equation}
|\Aring|_g^2 \, d\mu = |\Abring|^2_{\barg}\, d\bar{\mu}.
\label{AAbar}
\end{equation}
If $Y$ is minimal (rather than just Willmore) with respect to $g$, then $A = \Aring$ and 
\begin{equation}
|A|^2_g\, d\mu = |\Aring|^2_g \, d\mu = |\Abring|^2_{\barg} \, d\bar{\mu}
\label{keqkb}
\end{equation}
For brevity, the subscripts $g,\barg$ are often omitted when the meaning is clear. 

For any $Y\in {\cal M}$, the (Willmore) energy of $Y$ equals $\calE(Y):=\int_Y|A|^2\, d\mu$. In all the results below 
we restrict to the subset of surfaces $Y \in \calM_{k,g}$ with $\calE(Y) \leq M$ for some fixed $M < \infty$. 
We prove later that any complete $Y$ with finite energy necessarily meets $\del \HH^3$ orthogonally, so we
can omit this condition from the definition of elements of $\calM$. 

\medskip

\subsection{Results}
As explained earlier, we shall need to consider surfaces which satisfy a slightly stronger condition than finiteness
of Willmore energy. This involves a weighted version of the Willmore energy which we now define. 
\begin{definition}
\label{weight.fn}
Fix a number $N\in\mathbb{N}$. Given any finite set of points $\calO = \{O_1,\dots, O_N\}$, where each $O_j \in Y$, 
let $f_{\calO}(P):={\rm dist}(P, \calO)+5$. 
We call the points $O_k$ the poles of $f_\calO$. If $P\in Y$ and $O_k$ is the pole nearest to $P$, we write $P\sim O_k$.
We frequently write $f$ instead of $f_\calO$ for brevity; thus $f$ is the distance function from {\it some} set of 
$N$ points which may be anywhere  on $Y$. 

Now define the weighted energy 
\[
\calE_{p}(Y, \calO) := \int_Y|\Aring|^2f^{2p}_{\calO}\, d\mu;  
\]
we sometimes write this simply as $\calE_p(Y)$. 
\end{definition}

\begin{definition}
 \label{local.energy}
Fix $Y \in \calM$ and $\gamma = \del_\infty \overline{Y}$. Writing $B(P,R)$ as the open Euclidean half-ball 
centered at $P$ of radius $R$, for any $P \in \gamma$ and $R > 0$, then denote by $Y_{B(P,R)}'$ the path component
of $\overline{Y} \cap B(P,R)$ which contains $P$ in its closure 
and $\gamma'_{B(P,R)} = \overline{Y'_{B(P,R)}} \cap \del \RR^3_+$. 
The weighted and unweighted localized energies of $Y'_{B(P,R)}$ are given by 
\begin{equation}
\calE_p^{B(P,R)}(Y,\calO) :=\int_{Y'_{B(P,R)}} | \Aring|^2 f_\calO^{2p}\, d\mu, \quad \calE^{B(P,R)}(Y) :=\int_{Y'_{B(P,R)}} | A|^2 \, d\mu. 
\label{energy}
\end{equation}
\end{definition}


\begin{definition}
The $\zeta$-{\it Lipschitz radius} of a normalized, closed embedded $\calC^1$ curve $\gamma\subset \RR^2$
is defined as follows.  If $P\in \gamma$ and $\ell_\gamma(P) = T_P \gamma$, then let $\gamma_P
\subset \gamma$ be the largest open connected arc containing $P$ which is a graph over $\ell_\gamma(P)$.  Thus if
$P = 0$ and $\ell_P = \{(0,y,0)\}$, then $\gamma_P = \{(y,f(y)): a < y < b\}$ for some maximal $a < 0 < b$.  
We then define $\LipRad^\zeta_{\gamma}(P)$ to be the largest number $M$ such that $(-M,M)\subset (a,b)$ and 
$\frac{|f(y)-f(y')|}{|y-y'|}<\zeta$ for every $y,y'\in (-M,M)$.  Finally, we set 
\begin{equation}
\LipRad^\zeta(\gamma)=\inf_{P\in \gamma} \LipRad^\zeta_{\gamma}(P).
\end{equation}
Since $\gamma$ is compact, the easily verified lower semicontinuity of $\LipRad^\zeta$ implies that 
the infimum is attained at some point and $\LipRad^\zeta(\gamma)>0$. 
\end{definition}

\begin{theorem}
\label{eregularity}
There is a $\zeta_0 \in (0, 1/20)$ with the property that for any $\zeta \in (0,\zeta_0)$, there exists an $\e(\zeta,p)>0$ 
such that if $Y \in {\cal M}_{k,g}$ and $\calE^{B(P,R)}_p(Y)<\e$ for some $P\in \gamma =\del_\infty Y$ and $R \leq 1$, then
\[
\LipRad^\zeta_\gamma(Q)\ge \zeta\cdot \frac{R-|PQ|}{10}
\]
for all $Q\in \gamma'_{B(P,R)}$. 
\end{theorem}
From this and Lemma \ref{uniform.height} below, we can deduce the 
\begin{corollary}
\label{ereg.surface}
In the setting of Theorem \ref{eregularity}, there exists $\e'(\zeta,p)\le \e(\zeta,p)$ such that if $\calE^{B(P,R)}_p[Y] < \e'(\zeta,p)$, 
then $Y'_{B(P, R/2)}$ is a horizontal graph $z = u(x,y)$ over the half-disc $D(P,R/2)$ in the vertical half-plane 
$\RR_+ \times \ell_P$, and $|\nabla u|\le 2\zeta$ in $D(P,R/2)$. 
\end{corollary}

The Lipschitz radius is a reasonable measure of regularity on the space of normalized embedded curves $\gamma$. 
Note that if $\gamma_j$ is a sequence of such curves with $\LipRad(\gamma_j) \geq C > 0$, then there are uniform 
Lipschitz parametrizations around each point of every $\gamma_j$, hence in particular some subsequence of the 
$\gamma_j$ converge in $\calC^{0,\alpha}$ for any $\alpha < 1$ to a limit curve $\gamma$ which is itself Lipschitz. 

We also state, for future use, a slightly modified version of this result. Let $\calM'$ be the space of properly 
embedded Willmore surfaces $Y\subset \mathbb{H}^3$ meeting $\del \HH^3$ orthogonally and with 
with $\calC^1$ boundary curves $\gamma=\partial Y$. Thus $Y'_{B(P,R)}$ and $\calE_p^{B(P,R)}(Y)$ still make sense
if $Y \in \calM'$. The modification deals with surfaces $Y\in \calM'$ 
for which $\gamma'_{B(P,R)}$ intersects $\partial B(P,R)$.
\begin{theorem}
\label{eregularity2} 
For some $\zeta_0>0$ and every $\zeta \in (0,\zeta_0)$ there exists an $\e(\zeta,p)>0$ such that
if $Y \in \cal M'$, $\calE^{B(P,R)}_p\le \e(\zeta)$ and $\gamma'_{B(P,R)}$ intersects $\partial B(P,R)$ for some $P\in \gamma =\del_\infty Y$ 
 and if $\gamma'_{B(P,R)}$ is $\calC^1$ up to its endpoints, then 
\[
\LipRad^\zeta_\gamma(Q)\ge \zeta\cdot \frac{R-|PQ|}{10}
\]
for all $Q\in \gamma'_{B(P,R)}$. 
\end{theorem}

Theorem \ref{eregularity} leads to the following characterization of the possible limits of sequences of Willmore surfaces
with weighted energy bounded above. 
\begin{theorem}
\label{globalepsreg}
Let $Y_j \in \calM_{k,g}$ and suppose that $\calE_p(Y_j) \leq M$ for some $M > 0$. Suppose too that the distance 
between the various components of $\gamma_j = \del_\infty Y_j$ is uniformly bounded away from $0$. Then if
$0<\zeta\le \zeta_0$,  there is a subsequence, again relabelled as $Y_j$, which converges to a finite multiplicity 
(but possibly disconnected) Willmore surface $Y_*$ with boundary curve $\gamma_*$.
The convergence $Y_j\to Y_*$  is smooth away from $\{x=0\}$, except at a finite number of interior points, where $Y_*$ 
may fail to be smooth. In this limit, the set of poles $\calO^{(j)} \subset Y_j$ converges to a set of poles $ \calO^* \subset  \overline{Y_{*}}$.

Furthermore, there exist points $P_1, \ldots, P_{\Lambda} \in \gamma_*$, $\Lambda = \Lambda(\zeta)$, and corresponding sequences 
$P_i^{(j)} \in \gamma_j$, $i = 1, \ldots, \Lambda$, with $P_i^{(j)} \to P_i$ for all $i$, such that the convergence of 
$\gamma_j$ to $\gamma_*$ is $\calC^{0,\alpha}$ for every $\alpha < 1$ away from the points $P_i^{(j)}$. 
Finally, if $P\in \gamma\setminus \{P_1, \ldots, P_\Lambda\}$, then there is a 
line $\ell_P$ such that $Y_*$ is the graph of a Lipschitz function with Lipschitz constant $2\zeta$ over some disc in the 
half-plane $\RR^+ \times \ell_P$. 
\end{theorem}

The convergence in the interior relies on now standard results for Willmore surfaces, while the behaviour at the
boundary follows directly from Theorem \ref{eregularity} and Corollary~\ref{ereg.surface}. Indeed, let $\delta(Q)$
be the largest radius such that the half-ball $B(Q,\delta(Q))$ centered at $Q\in\gamma_j$ satisfies 
$\calE^{B(Q,\delta(Q))}_p(Y'_j ) \le  \e'(\zeta)$. Letting $t$ be the arc-length parameter on any component of $\gamma_i$, 
then the upper bound on $\calE_p(Y_j)$ implies that for all but those finitely many 
values of $t$ corresponding to the points $Q_i^{(j)}$, $\liminf_j\delta(\gamma_{j}(t))>0$. 
From this and a diagonalization argument we deduce the asserted convergence near the boundary. 
The interior convergence follows from the usual interior $\epsilon$-regularity results for smooth Willmore surfaces, 
(see Theorem I.5 in \cite{riv1}) together with a covering argument and a further diagonalization. 
Hence, up to a subsequence, the $Y_j$ converge to a Willmore surface $Y_*$ which is smooth away from a finite number 
of points.  It remains to prove that $Y_*$ has finite multiplicity. However, if some component of $Y_{*,a}$ of $Y_*$ 
has infinite multiplicity, then the upper bound on energy implies that $Y_*$ must be totally geodesic,
which would force the length of $\del Y_{*,a}$ (counted with multiplicity) to be infinite. This is impossible. 
 
Our next result shows that the limit curve $\gamma_*$ is $\calC^1$, rather than just Lipschitz, away
from a finite set of points.
\begin{theorem}
\label{removability} In the setting of Theorem \ref{globalepsreg}, the curve $\gamma_* = \del_\infty Y_*$ is 
piecewise $\calC^1$, with singularities occuring (at most) at the set $\{P_1, \ldots, P_\Lambda\}\subset \gamma_*$. 
\end{theorem}
\begin{remark}
A modification of the proof of Lemma~\ref{uniform.height} below shows that $\overline{Y}_*$ is then $\calC^1$ up 
to $\gamma_* \setminus \{P_1, \ldots, P_\Lambda\}$. 
\end{remark}

We also describe bubbling in this setting by showing that away from points where the convergence $\gamma_j\rightarrow\gamma_*$ 
is not $\calC^1$, the loss of compactness is due to some portion of the Willmore suraces {\it with non-zero energy} 
escaping to infinity: 
\begin{theorem}
 \label{jump.implies.bubble}
Let $Y_j$ be a sequence in $\calM_{k,g}$, $\calE_p(Y_j) \leq M<\infty$, with $Y_j \to Y_*$ where $Y_*$ is $\calC^1$ up 
to $\gamma_* \setminus \{P_1, \ldots, P_\Lambda \}$. After rotating and translating, we write each $Y_j$ as a horizontal 
graph $z = u_j(x,y)$ over the half-disc $\{x^2+y^2\le \delta^2\}$, with $|\nabla u_j|\le 2\zeta$ and $u_j\rightarrow u_*$ 
in $\calC^\infty$ away from $\{x=0\}$ and in $\calC^{0,\alpha}$ up to $\{x=0\}$. Suppose too that for some $y_0\in 
(-\delta,\delta)$, $\lim_{j\to \infty}\partial_y u_j(y_0,0)\ne \partial_y u_*(y_0,0)$. Then there exists a sequence of interior 
points $Q_j\in Y'_j\bigcap B(0,\delta)$ with $Q_j\rightarrow A_j := (0,y_0, u_j(y_0,0))$ and a sequence of
hyperbolic isometries $\psi_j$ mapping $Q_j$ to $(1,0,0)$ so that $\psi_j(Y_j)\to Y'_*$ for some complete 
Willmore surface $Y'_*$ with $\calE(Y'_*) > 0$. 

At most $N$ non-isometric blow-ups can be obtained in this way, and there exists a number $M'>0$ 
such that for each sequence $Q_j$ there is a sequence of poles $\calO^{(j)}\in Y_j$ such that 
${\rm dist}_{Y_j}(Q_j,O^{(j)}) \le M'$.
\end{theorem}

This Theorem actually proves that the convergence $\del Y_j\to \del Y_*$ is $\calC^1$ near all 
points $P\in \partial Y_*\setminus (\partial Y_*\bigcap \{P_1,\dots, P_\Lambda\})$ {\it except} at those points 
on $\del_\infty Y_*$ which are limits of the poles $O_i^{(j)}$:

\begin{corollary}
\label{no.jump.no.poles}
Assume that for some point $P\in \partial Y_*\setminus (\partial Y_*\bigcap \{P_1,\dots, P_\Lambda\})$ there exists a relatively open set $\Omega\subset \overline{\mathbb{R}^3_+}$
such that $O_i\notin\Omega$ for all poles $O_i\in Y_j$ and for all $j$ large enough. 
Then the curves $\partial Y_j$ converge to $\partial Y_*$ in the ${\cal C}^1$ norm in the domain $\Omega$. 
\end{corollary}

Indeed, if this were not the case at some point $Q\in \partial Y_*\bigcap \Omega$, then for $p' \in (1,p)$ and any
$\epsilon>0$ there exists an open half-ball $B(P,\delta)$  such that 
$\calE_{p\rq{}}^{B(P,\delta)}(Y_j)\le \epsilon$ for all $j$ large enough.
If $\epsilon$ is small enough, this implies a lower bound on ${\rm LipRad}(\gamma_j)$ in the 
half-balls $B(P,\delta/2)$. 

Applying Theorem \ref{jump.implies.bubble}, we obtain a blow-up limit $\varphi_j(Y_j)\to Y\rq{}_*$ 
with $\calE(Y\rq{}_*)>0$, where the $\varphi_j$ are hyperbolic  isometries centered at $Q\in\partial_\infty\mathbb{H}^3$. 
The $\calC^\infty$ convergence away from the boundary of $\varphi_j(Y_j)$ implies that there exist balls 
$B(P_j,1)\subset Y_j$ of (intrinsic) radius $1$ in  $Y_j$ with $P_j\to Q$ and with $\calE^{B(P_j,1)}(Y_j)\ge\epsilon_0 > 0$.
(Indeed, it suffices to let $\epsilon_0<\calE^{B(P,1)}(Y\rq{}_*)$ where $B(P,1)\subset Y\rq{}_*$ is
any intrinsic ball where the energy is non-zero). Now, $P_j\to Q$ readily implies that $f_{j,{\calO}}(Q_j) \to\infty$
($f_{j,{\calO}}$ is the weight function for the surface $Y_j$), so that $\calE^{B(P,\delta)}_{p\rq{}}(Y_j)\to\infty$. This is a contradiction.  

In the last section of this paper, we construct examples where bubbling to infinity does occur. These are sequences 
of minimal (and thus Willmore) surfaces $Y_j \in \calM$ with $\calE(Y_j) \leq M<\infty$ and with 
$\overline{Y}_j$ converging smoothly away from a finite number of points on the boundary. At these points, 
the convergence fails to be $\calC^1$, despite the fact that the curves $\gamma_j$ and $\gamma_*$ are 
all $\calC^\infty$. 

\medskip

\noindent {\bf Acknowledgements:}  We offer special thanks to Tristan Rivi\`ere for generously sharing his insight 
into these questions and for much encouragement during the early stages of this work. The first author acknowledges
helpful conversations with Jacob Bernstein. R.M. was supported by NSF Grant DMS-1105050. S.A. was supported 
by  NSERC grants 488916 and 489103, as well as  Clay  and Sloan fellowships. 

\section{Some geometric lemmas}
We begin with some geometric results, pertaining primarily to {\it fixed} complete  Willmore surfaces $Y$ with 
$\calE(Y) < \infty$ and with $\gamma = \del_\infty Y$ a finite union of compact embedded Lipschitz curves. 
We prove first that any such $Y$ meets $\del_\infty \HH^3$ orthogonally, which is the well-known 
behaviour when $Y$ is $\calC^2$ up to the boundary, We then show that if some segment of $\gamma$ is graphical 
with a bounded Lipschitz constant, then a portion of the Willmore surface directly above this segment is also 
graphical, with bounded gradient. The proofs are almost entirely geometric, involving blowup arguments, though 
we rely on one analytic fact which is the $\e$-regularity theorem for (interior) Willmore discs. The finiteness
of energy is used crucially at several places. 

\begin{lemma}
\label{asympt.vertical}
Let $Y$ be a complete properly embedded Willmore surface such that $\gamma = \del_\infty Y$ is a finite collection 
of embedded Lipschitz curves. Let $P_j$ be a sequence of points in 
$Y$ converging to a point on $\del_\infty Y$. If $\bar{\nu}$ denotes the Euclidean unit normal to $Y$, then 
$\langle \partial_x, \bar{\nu} \rangle_{\olg}(P_j)\rightarrow 0$ as $j \to \infty$.
\end{lemma}
\begin{proof}  Since $Y$ has finite energy, then $\int_{Y\cap \{x\le C x(P_j)\}} |A|^2\, d\mu\rightarrow 0$ for any $C > 0$.  
Now suppose that the assertion is false. Thus, passing to a subsequence if necessary, 
$\langle \partial_x,\bar{\nu}\rangle_{\olg}(P_j)\to \beta\ne 0$. Let $B_1(P_j)$ be the ball of radius $1$ around $P_j$ 
with respect to the metric $g$. Passing to a further subsequence, we may assume that $B_1(P_j) \cap B_1(P_k) = 
\emptyset$ for $j \neq k$. Then $\int_{B_1(P_j)} |A|^2\, d\mu\to 0$, since otherwise $\calE(Y)$ would be infinite. 

Now translate $Y$ horizontally and dilate by the factor $1/x(P_j)$ so that $P_j$ is mapped to $(1,0,0)$ and denote by
$Y_j$ the resulting sequence of surfaces. Since each $Y_j$ passes through the fixed point $(1,0,0)$ and
$\int_{Y_j\cap \{ x\le M\}} |A|^2\, d\mu\rightarrow 0$ for any $M > 0$,  we can invoke the a priori pointwise bounds 
for $|\nabla^p A_j|$ on a ball of any fixed radius around this fixed point using \cite[Theorem I.5]{riv1}. These
show that yet a further subsequence of the $Y_j$ converge in the $\calC^\infty$ topology on compact sets to a complete 
Willmore surface $Y_*$.

Since $\calE^{B_R(P_j)}(Y_j) \to 0$ for any $R > 0$, we see that $Y_*$ is totally geodesic, and hence is either
a vertical plane or a hemisphere; its slope at $(1,0,0)$ equals $\beta \neq 0$, so we must be in the latter case.  
This shows that there is a fixed constant $R = R(\beta) > 0$ such that if $j$ is 
large, then the ball $B_R(P_j)$ in $Y$ contains a point $Q_j$ where $T_{Q_j} Y$ is horizontal, i.e.\ 
parallel to $\{x=0\}$. 

We can assume (passing again to a further subsequence) that $x(Q_j)$ is strictly monotone decreasing, so 
a standard minimax argument shows that we may choose a sequence of points $Q_j' \in Y_j$ which are 
critical points of index one for the function
$x$. In other words, writing $Y_j$ as a graph $x = v(y,z)$ near $Q_j'$, then $v$ has a saddle at $Q_j'$.  We can
therefore translate horizontally and dilate by the factor $1/x(Q_j')$ to obtain a sequence $Y_j'$ of Willmore 
surfaces which converge locally in $\calC^\infty$ to a complete Willmore surface $Y_*'$ passing through the point $(1,0,0)$. 
By construction, for any $M>0$, 
\begin{equation}
\label{energy.to.zero}
\int_{Y'_j\cap \{x\le M \} }|A_j|^2\, d\mu_j\rightarrow 0.
\end{equation}
Using the interior curvature estimates that follow from the $\epsilon$-regularity in \cite{riv1}  again, 
we see that the convergence of $Y_j'$ to $Y_*'$ is $\calC^\infty$ 
near the point $(1,0,0)$, hence $Y_*'$ has a horizontal tangent plane at this point. Furthermore, the two 
principal curvatures at $Q_j'$ relative to the ambient Euclidean metric are $\kappa_1 \geq 0$ and 
$\kappa_2 \leq 0$, and these inequalities must persist in the limit. This means that $Y_*'$ cannot be a hemisphere. 
However, \eqref{energy.to.zero} implies that $\calE(Y'_*) = 0$, which yields a contradiction.  \end{proof}

An almost identical argument proves the 
\begin{lemma}
\label{no.spheres}
Let $Y$ be a fixed Willmore surface in $\HH^3$ with $\calE(Y) < \infty$.  Let $P_j$ be a sequence of
points in $\del_\infty Y$ and choose $\delta_j \searrow 0$. Denote by $B^+_{\delta_j}(P_j)$ 
the Euclidean half-ball centered at $P_j$ and with radius $\delta_j$. Assume that the sequence
of dilated translates $\delta_j^{-1} ( Y \cap B_1^+(P_j) - P_j)$ converges to a Willmore surface
$\tilde{Y}$. Then $\tilde{Y}$ must be a vertical half-plane. 
\end{lemma}

We next turn to proving local graphicality of any Willmore surface of finite energy near points where the 
boundary curve is Lipschitz. 

\begin{lemma}
\label{Lipschitz.lift}
Let $Y$ be a complete properly embedded Willmore surface in $\HH^3$ with finite energy and such that 
$\del_\infty Y = \gamma$ is a finite union of closed embedded rectifiable loops. 
Define $\calS_B$ to be the set of all points $P \in \gamma$ for which there exists a connected subarc 
$\gamma_P \subset \gamma$ which is a graph over a straight line $\ell_P \subset \RR^2$ containing $P$ 
which,  if we rotate and translate so that $\ell_P$ is the $y$-axis, has graph function $z=f(y)$,
$|y|\le \delta(P)$, satisfying $\mathrm{Lip}(f)\le B$. 

Then there exists an $h>0$, independent of $P$, such that the portion $Y'_{B(P,h \, \delta(P))}$ of the surface $Y$
is graphical over the half-disc $\{\sqrt{x^2+y^2} \le h \, \delta(P), \ z = 0\} $ with graph function $z=u(x,y)$, where
$u$ satisfies $|\nabla  u|\le 2B$. 
 \end{lemma}
\begin{proof}  If this were false, then there would exist a sequence $P_j\in\gamma$, lines $\ell_j$ and 
graph functions $f_j:[-\delta_j,\delta_j]\rightarrow \RR$ for $\gamma$ with Lipschitz constant $B$, and sequences
of numbers $h_j\to 0$ and points $Q_j\in Y'_{B(P_j, h_j\delta_j)}$ with coordinates $(x_j, y_j, z_j)$ (using coordinates $(x,y,z)$ 
where $P_j$ is the origin and $\ell_{P_j}$ is the $y$-axis), such that angle between the unit 
(Euclidean) normal $\bar{\nu}(Q_j)$ to $Y$ at $Q_j$ and $\del_z$ is greater 
than $ \arctan(2B)$. 

Since $Y$ has finite energy, we have that $\calE^{B(P_j, h_j\delta_j)}(Y) \to 0$, so a contradiction can
be drawn by a blow-up argument. Translate so that $y_j=0$, then dilate by the factor $\frac{1}{x_j}$ 
to obtain a sequence of surfaces $\tilde{Y}_j$. By construction, $ \del_\infty \tilde{Y}_j$ is graphical over the $y$-axis 
at least over the interval $|y|\le \frac{1}{h_j}$, with Lipschitz constant $B$. Furthermore, the angle 
between $\bar{\nu}$ and $\del_z$ at $(1,0,0)$ is greater than $\pi/2-\arctan(2B)$. However, $Y_j$ 
converges to a vertical half-plane $Y_*$, and since the 
convergence is $\calC^\infty$ away from the boundary by \cite{CS}, the angle condition at $(1,0,0)$ is  preserved in the limit. 
However by Lemma \ref{no.spheres}, from the Lipschitz bound on the graph 
function $f_j$, we see that $Y_* = \{z= \alpha y+ \beta, x > 0\}$ 
for some $\alpha$ with $|\alpha|\le B$. This contradicts the angle condition at $(1,0,0)$. 
\end{proof}

We also need a slight variant of this. 
\begin{lemma}
\label{uniform.height} Consider a sequence of complete Willmore surfaces $Y_j$, the closures of which pass through 
the $(0,0,0)$. Assume that the subdomains $Y'_{j,B(0,3)}$ satisfy $\calE^{B(0,3)}(Y_j)\to 0$,
and that $\gamma_j=\partial_\infty Y'_{j,B(0,3)}$ is a graph $z = f_j(y)$ over the interval $|y| \leq 2$
with $f_j \in \calC^1$ and $|f'_j(y)|\le \delta$ for some $\delta>0$. 
Then there exists an $\e_0 (\delta)> 0$  such that $Y'_{j,B(0,3)}$ is a graph $z = u_j(x,y)$ 
over the rectangle $\calR :=\{ 0\le x\le  \e_0(\delta), |y| \leq  2\}$, and $|\nabla u_j|\le 2\delta$ on $\calR$. 
\end{lemma}
\begin{proof}
This is proved essentially as before. We pick $\e_0(\delta)$ small enough so that 
all functions $u_j(x,y)$ with $(x,y)\in [0,\e_0(\delta)]\times [-2,2]$  
whose graphs are portions of upper half-spheres and satisfy $|\partial_yu_j|_{x=0}|\le\delta$ 
also satisfy $|\nabla u_j(x,y)|\le \frac{3\delta}{2}$ for $(x,y)\in [0,\e_0(\delta)]\times [-2,2]$. 
If the claim were to fail for this $\e_0(\delta)$, we could choose points $P_j \in Y_j$ contained in the portion of $Y_j$ 
which is graphical over this rectangle where 
$|\nabla u_j(P_j)| > 2\delta$. Since $Y_j \to Y_*$ smoothly away from $x=0$ and $Y_*$ must be a portion 
of a hemisphere,  we must have $x(P_j)\to 0$.  Now dilate by the factor $x(P_j)^{-1}$; this produces a sequence of Willmore
surfaces $\tilde{Y}_j$ which converge to a vertical half-plane $Y_*$ which meets the $xy$-plane at 
a small angle bounded above by $|\arctan(\delta)|$, but such that the corresponding
graph functions $\tilde{u}_j$ satisfy $|\nabla \tilde{u}_j| > 2\delta$ at a fixed point $(0,0,1)$. However,
convergence in this dilated setting is still smooth away from $x=0$ by \cite{riv1}, so this is a contradiction. 
\end{proof}

\section{The $\e$-regularity results: Small energy controls boundary regularity.}
\subsection{Vanishing energy implies $\calC^1$ boundary convergence.}
We first state a key proposition, and then deduce Theorems~\ref{eregularity} and \ref{eregularity2}  from it. 

For any $\zeta \in (0,1]$, consider the (unique) circle $C^\zeta_*$ in the $yz$-plane 
which is tangent to the $y$-axis at the origin and whose
 graph function $f^\zeta_*(y)$ over the interval $[-1,1]$ satisfies $(f^\zeta_*)'(1)=\zeta$.
Pick $\zeta_0$ small enough so that for each $\zeta\in (0,\zeta_0]$ the circle $C^\zeta_*$
 is contained in the open ball $B(0, \frac{5}{\zeta})\subset \RR^2$.

\begin{proposition}
\label{theclaim} Let $\zeta$ and $\zeta_0$ be as above. Suppose that $Y_j $ is a sequence of connected 
Willmore surfaces in $\HH^3\bigcap B(0,2)$ with boundaries at infinity $\del_\infty Y_j = \gamma_j$,
and the remaining boundary components on the outer boundary of this half-ball. Assume $\calE(Y_j)\le M<\infty$.
 We assume furthermore that: 
\begin{enumerate}
\item[a)] Each $\gamma_j$ is the graph of a function $f_j$ over $[-1,1]$, which satisfies $|f_j(y)-f_j(y')|\le \zeta|y-y'|$ for all $y, y' \in [-1,1]$
and $f_j(0)=0$, $f'_j(0)=0$, $f'_j(1)=\zeta$;
\item[b)] $\LipRad^\zeta_{\gamma_j}(P) \ge \frac{2-|P|}{A}$, for some fixed $A>0$;
\item[c)] ${\cal E}^{B(0,2)}_p(Y_j)\to 0$ as $j\to\infty$.
\end{enumerate}
Then $f_j\rightarrow f^\zeta_*$ in $\calC^1([-1,1])$. 
\end{proposition}
In other words, if the weighted energies of a sequence of Willmore surfaces converge to zero in some fixed half-ball, and if the boundaries 
at infinity of these Willmore surfaces are uniformly Lipschitz in the qualitative sense above, then these boundaries must converge 
to a particular circular arc defined by the normalization, and the convergence is actually in $\calC^1$.

For future reference, we state another proposition which guarantees $\calC^1$ convergence of boundary curves 
under slightly different assumptions on the boundary curves. 
This will be used in the proof of the Theorem \ref{removability} above. 

\begin{proposition}
\label{ereg.prp2}
Assume that $Y_j$ is a  sequence of connected Willmore surfaces in $\HH^3\bigcap B(0,2)$, with boundaries 
at infinity $\del_\infty Y_j = \gamma_j$, and with all other boundaries contained in the outer boundary
of the half-ball $B(0,2)$. Assume $\calE(Y_j)\le M<\infty$. Assume further that:
\begin{enumerate}
\item[a)] $Y_j$ is the graph of a function $z=u_j(x,y)$ over the half-disc $\{x^2+y^2\le 2, z=0\}$.
\item[b)]$|\nabla u_j|\le 2\zeta\le 1/10$ for $x>0$ and $f_j(y):=u_j(0,y)$ is a Lipschitz function with Lipschitz constant $\zeta$. 
\item[c)]  $\calE_p(Y_j)\rightarrow 0$, and $Y_j$ converges to the upper half-disc $\{z=0, x^2+y^2< 2\}$. 
\item[d)] All $f_j$ are differentiable at $y=0$. 
\end{enumerate}
Then $\lim_{j\rightarrow \infty} f'_j(0)=0$.
\end{proposition}

\subsection{Proposition~\ref{theclaim} implies $\e$-regularity} 
We now show that Theorems~\ref{eregularity} and \ref{eregularity2} can be deduced from Proposition~\ref{theclaim}. 

The argument is by contradiction. Assume that for every $j \geq 1$ there exist surfaces $Y_j\in \calM$, 
points $P_j\in \gamma_j:=\del Y_j$ (and radii $R_j \leq 1$ in the context 
of Theorem \ref{eregularity}) such that $\calE^{B(P_j,R_j)}_p(Y_j)<\frac{1}{j}$, 
yet $\LipRad^\zeta_{\gamma_j}(Q_j)<\zeta \frac{R_j-|P_jQ_j|}{10}$ for some $Q_j\in \gamma_j \cap B(P_j,R_j)$. Observe 
that the points $Q_j$ must lie in the {\it open} ball $B(P_j,R_j)$ since $\LipRad^\zeta(\gamma_j) > 0$.

Select a point $Z_j$ in the open ball $B(P_j,R_j)$ so that 
\[
\inf_Q \, \frac{\LipRad^\zeta_{\gamma_j}(Q)}{(R_j - |P_j Q|)} = \frac{\LipRad^\zeta_{\gamma_j}(Z_j)}{(R_j - |P_j Z_j|)},
\]
and note that this ratio is less than $\zeta/10$. Let $\delta_j:=\LipRad^\zeta_{\gamma_j}(Z_j)$.  By translation and rotation,
assume that $Z_j = 0$ and $T_{Z_j}\gamma_j$ is the $y$-axis. Now dilate by $\delta_j^{-1}$.  
Denote the rescaled surface by $\tilde{Y}_j$ and the rescaled boundary curve by $\tilde{\gamma}_j$;
note that $|\tilde{\gamma}_j| = 100\pi\delta_j^{-1}$. Thus $\tilde{\gamma}_j$ is a graph $z=f_j(y)$ over $[-1,1]$, 
with $f_j(0)=0$, $f'_j(0)=0$ and $|f_j(y)-f_j(y')|\le\zeta |y-y'|$.  Moreover, because $[-1,1]$ is the maximal interval 
on which the Lipschitz norm of $f_j$ is bounded by $\zeta$, 
we must have either $|f'_j(-1)|=\zeta$ or $|f_j'(1)| = \zeta$, and to be definite we suppose that $f_j'(1) = \zeta$ for each $j$. 

The translated and rescaled ball $\tilde{B}_j$ contains $B(0,\frac{5}{\zeta})$. Furthermore, by the choice of $Z_j$ and the dilation, 
we see that there exists an $\eta>0$ such that  for each $P\in \tilde{\gamma}_j\bigcap B(0,\frac{5}{\zeta})$,
$\LipRad^\zeta_{\tilde{\gamma}_j}(P)\ge \eta$. 

We claim that $\tilde{\gamma}_j \to C^\zeta_*$ in $\calC^1$. Assuming this for the moment, 
we show that this leads to a contradiction in Theorems \ref{eregularity} and \ref{eregularity2}.

For Theorem \ref{eregularity2}, the contradiction is immediate. Indeed, the curves $\tilde{\gamma}_j$
intersect the circle $\partial \tilde{B}_j$, which contradicts the fact that $\tilde{\gamma}_j
\rightarrow C^\zeta_*$, which lies strictly in the interior of $\tilde{B}_j$. 

As for Theorem \ref{eregularity}, let $g_j: [-50\pi\delta_j^{-1}, 50\pi\delta_j^{-1}] \to \RR^2$ 
parametrize $\tilde{\gamma}_j$ by arclength, so the length along the curve between 
$g_j(0)$ and $g_j(s)$ is $|s|$; similarly, let $g_*: [-50\pi\delta_j^{-1},50\pi\delta_j^{-1}]\rightarrow C_*^\zeta$ 
be a (multi-covering) arclength parametrization of $C^\zeta_*$. 

Observe that $|C^\zeta_*|=2\pi R_\zeta$, 
with $R_\zeta=\sqrt{1+\frac{1}{\zeta^2}}$.  
Our claim gives that $g_j(s)\rightarrow g_*(s)$ in $\calC^1([-\pi R_\zeta,\pi R_\zeta])$, so in particular, 
\[
\lim_{j\to \infty} g_j(-\pi R_\zeta) = \lim_{j \to \infty} g_j(\pi R_\zeta), \ \ 
\lim_{j\to \infty} g'_j(-\pi R_\zeta) = -\lim_{j\to \infty}g_j'(\pi R_\zeta).
\]
This shows that $\lim_{j\to \infty} |\tilde{\gamma}_j|=|C^\zeta_*|$. On the other hand, we know that $|\tilde{\gamma}_j|
=100\pi\delta_j^{-1}$, which is impossible since $|C_*^\zeta|=2\pi R_\zeta$. 

\medskip

\noindent {\it Proof that Proposition~\ref{theclaim} implies $\gamma_j\rightarrow C^\zeta_*$ in $\calC^1$:} Let $2\zeta'$
be the length of the arc in $C^\zeta_*$ which is  a graph over the interval $y\in [-1,1]$.
By Proposition~\ref{theclaim}, $g_j(s)\rightarrow g_*(s)$ for $s\in [-\zeta', 
\zeta']$. Let $M>0$ be the largest number in $[0,\pi R_\zeta]$ such that
$g_j(s)\rightarrow g_*(s)$ in $\calC^1([-M,M])$.  We must prove that $M=\pi R_\zeta$, and 
moreover, for any small $\e >0$ and $j$ sufficiently large, that there exists an $\e_j > 0$ with
$\lim_{j \to \infty} \e_j = \e$ and $g_j(-\pi R_\zeta-\e_j)=  g_j(\pi R_\zeta-\e)$. 
The first claim ensures that $\gamma_j([-\pi R_\zeta,\pi R_\zeta]) \to C_*^\zeta$, 
while the second implies that $\gamma_j(s)$ closes up on a small extension of the interval 
$[-\pi R_\zeta,\pi R_\zeta]$. 

The first part is proved by contradiction: Assume $M<\pi R_\zeta$, and consider the pairs $(g_j(M), g'_j(M))$.
These converge to $(g_*(M), g'_*(M))$, so for $j$ large they lie in the open set $\calU$ where $\LipRad$
is bounded below by some $\eta>0$. Now let $\ell_j$ be the tangent line to $\tilde{\gamma}_j$ at $g_j(M)$. Consider the intervals 
of length $\eta$ centered at $g_j(M)$ on each $\ell_j$.  After translation, rotation and dilation by the factor $\eta^{-1}$,
the rescaled $\tilde{\gamma}_j$ can be written as the graphs of functions $\phi_j$ on $[-1,1]$. Applying
 Proposition~\ref{theclaim} to these functions, we see that $\phi_j\to f_*$ 
in $\calC^1$. Hence $g_j\rightarrow g_*$ on a larger interval $[-M', M']$, which contradicts
the maximality of $M$. 

As for the second part of the claim, note that the argument above shows that for $|\tau|\le \eta$ we have
\[
\lim_{j\to \infty} (g_j(-\pi R_\zeta-\tau),g'_j(-\pi R_\zeta-\tau))= \lim_{j\to\infty} 
(g_j(\pi R_\zeta-\tau),-g'_j(\pi R_\zeta-\tau)),
\]
because of the lower bound $\LipRad^\zeta\ge\eta$ and the $\calC^1$ convergence of the $g_j$ on 
$[-\pi R_\zeta-\tau, -\pi R_\zeta+\tau]$ to an arc of $C_*^\zeta$. 

Now, assume that for some fixed $\e>0$, there exists a subsequence in $j$ such that $g_j(-\pi R_\zeta-\epsilon)
\ne g_j(\pi R_\zeta-s)$ for any $s\in (0,2\e)$. In particular this says that $g_j(t)$ does not ``close up'' 
for $t\le -\pi R_\zeta$ and $t\ge \pi R_\zeta$. 

This gives a sequence of values $\tau_j\in [-50\pi\delta^{-1}_j, -\pi R_\zeta]\bigcup [\pi R_\zeta, 
50\pi\delta^{-1}_j]$ such that $\tau_j\to \tau_*$, $g_j(-\tau_j)\to P$, with $P\in C^\zeta_*$, yet 
$g'_j(\tau_j)\to T_*$ for some vector $T_*$ which is transverse to the tangent vector $T$ of $C^\zeta_*$ at $P$.  
However, if $j$ is large enough, then $g_j(\tau_j)\in {\cal U}$, and hence 
$\LipRad^\zeta_{\gamma_j}(g_j(\tau_j))\ge \eta>0$. But this implies that $\gamma_j$ must 
{\it self-intersect} near $P$, which contradicts that the boundary curves are embedded. $\Box$

\subsection{An overview of the strategy}
In the next two sections, we prove Propositions \ref{theclaim} and \ref{ereg.prp2}. In a nutshell, both results show,
in slightly different settings, that if the weighted energies of portions of the  Willmore surfaces $Y_j\subset \HH^3$ converge to zero, 
then $\gamma_j=\partial_\infty Y_j$ must converge {\it in the $\calC^1$ norm} to the boundary curve $\gamma_*$ of 
a totally geodesic surface $Y_*$.  We stress that the convergence of the graphical 
portions of the sufaces $Y_j$ to $Y_*$ is $\calC^\infty$ away from
$\{x=0\}$; the novelty here is the $\calC^1$ convergence at the boundary.  

Since the argument has several steps, we now provide a moderately detailed outline of the strategy. 
If the results were false, we could find a sequence of Willmore surfaces $Y_j$ satisfying
the hypotheses but for which the $\calC^1$ convergence fails at some boundary point. 
Thus, having written the boundary curves graphically, we assume that there exists $y_0\in [0,1]$ such that 
$\lim_{j\to \infty} f'_j(y_0)= b_1 \neq b_2 = f'_*(y_0)$. Because the local energy converges to zero, the limit
$Y_*$ is totally geodesic, and the convergence is $\calC^\infty$ away from $\{x=0\}$.  Furthermore,
at $\{x=0\}$, $f_j \to f_*$ in $\calC^\beta$ where the graph of $f_*$ is a circular arc. 

Compose with a suitable sequence of rotations, reflections  and inversions so that we can 
assume that $(y_0, f_j(y_0)) = (0,0)$ and (maintaining the names of all surfaces and curves)
that $Y_*$ is a portion of the vertical plane $\{z=0\}$. By assumption b) of Proposition \ref{theclaim},
each ${\gamma}_j$ is the graph of a functions $f_j$ defined on a fixed interval $[-1,1]$, 
and the limiting curve $\gamma_*$ is the graph of $f_*=0$ on this same interval. The hypothesis is
that $\lim_{j\to \infty}f_j'(0) = \alpha > 0$, although $f_*'(0) = 0$.

The argument proceeds in two steps. We first show that there exists a sequence of hyperbolic isometries 
$\varphi_j$ such that the surfaces $\varphi_j(Y_j)$ satisfy all the assumptions of Propositions \ref{theclaim}
and \ref{ereg.prp2} (including the jump in the limit of the first derivatives), but so that some fixed portion 
of $\varphi_j(Y_j)$ is covered by isothermal coordinates, the associated conformal factor of which is uniformly
bounded. This construction relies crucially on ideas in \cite{DLM}, many of which go back to the influential 
paper \cite{MS}.  The work here will involve modifying
some arguments in \cite{DLM}, which is possible because of some special features of our setting,
to ensure that the jump in the first derivative has a fixed size $\alpha-\beta>\frac{\alpha}{2}$.

However, we then use particular properties of these isothermal cordinate systems to prove that no such
jump in the limit of the first derivatives can occur. Writing $\varphi_j(Y_j)$ as the graphs of functions $u_j$, 
and denoting the isothermal coordinates by $(q_j, w_j)$, the idea is to control $\partial_{w_j} u_j|_{(0,0)}$
using that $\del_{w_j}u_j\to 0$ as $j\to \infty$ uniformly along $\{x=1\}$. The relationship between
these derivatives at $x=0$ and $x=1$ is obtained using two integrals, the first of the mixed component 
of the second fundamental form of $\varphi_j(Y_j)$ with respect to the Euclidean metric, and the second
depends on a derivative of the conformal factor.  We show that these integrals are bounded in terms of 
$\calE_p(\varphi_j(Y_j))$ and hence converge to $0$.  
The estimate for the first integral uses a realization of Willmore surfaces as harmonic
 maps into the $(3+1)$-dimentional deSitter space. 
In Euclidean coordinates, the energy integrand for this map turns out precisely to be the trace-less second 
fundamantal form $|\Abring|^2$. This, together with the harmonic map 
equation and the strong subharmonicity of the distance function on our surfaces 
yield bounds on $|\Abring|$ which are {\it integrable} in $x$. It is at this point exactly that the boundedness of 
{\it weighted} Willmore  energy (as opposed to the regular Willmore energy) is used. 
The control of the second integral
follows from interpreting it as one term in a flux formula whose interior term is controlled by $\calE(\varphi_j(Y_j))$.
 
\begin{remark}
The jump of the first derivative of $\gamma_j$ can also be described in terms of the Euclidean coordinate function 
$z$ restricted to the surface $Y_j$. Indeed, the jump condition is the same as 
\begin{equation}
\label{jump'} 
|\lim_{j\to\infty}\bar{\nu}_j(z)-\bar{\nu}_*(z)|=\alpha>0, 
\end{equation}
where $\bar{\nu}_j$ and $\bar{\nu}_*$  are the Euclidean unit tangent 
vectors to $\del_\infty Y_j$  and $\del_\infty Y_*$ at $(0,0,0)$. 
\end{remark}

\section{Uniform isothermal parametrizations}
\label{uniform.isothermal}
We now choose a sequence of  hyperbolic isometries $\varphi_j$ which map the surfaces $Y_j$ to a new sequence
of surfaces which satisfy the assumptions of our propositions (in particular they converge to a vertical half-plane) but 
such that some fixed  portions of these rescaled surfaces admit isothermal coordinates $(q_j,w_j)$, the conformal factors
of which are uniformly bounded in $\calC^0$, $W^{2,1}$ and $W^{1,2}$. We must also ensure that the transformed surfaces 
still exhibit a jump in first derivative at the origin.

Let us put this into context. In their well-known paper \cite{MS} (see also \cite{H}),
M\"uller and Sverak show that surfaces of finite total curvature in $\RR^3$ admit isothermal parametrizations 
with bounded conformal factors.  This argument was extended by DeLellis and M\"uller \cite{DLM} to obtain 
a uniformization of spheres $\Sigma\subset \RR^3$, proving in particular that if $\int_\Sigma |\Abring|^2\, d\mu< 8\pi$,
then one can find a conformal map $\Psi:S^2\to \Sigma$ for which the conformal factor is controlled by this energy.
In order to deal with the noncompactness of the conformal group, which implies the existence of such 
maps with conformal factor having arbitrarily large supremum, they impose a normalizing balancing condition 
for the total curvature restricted to certain hemispheres.  Our argument below follows this idea, 
with one key difference: we obtain this balancing not intrinsically by precomposing with conformal maps of $S^2$, but
extrinsically, by post-composing by M\"obius transformations of $\RR^3$. This is necessary for ensuring that the 
jump in first derivative still occurs in these isothermal coordinates.  
This extrinsic balancing may not be viable for arbitary spheres $\Sigma\subset \RR^3$ close to $S^2$,
but the surfaces we consider already have two reflection symmetries so only one extra balancing condition 
needs to be attained.
 
This construction is used in the proofs of Propositions \ref{theclaim} and \ref{ereg.prp2} in slightly different settings,
so we prove the present result in two different settings as well. These involve different hypotheses on the 
boundary curves $\gamma_j = \del_\infty Y_j$ (assumed as always to pass through the origin). 
The $\gamma_j$ are $\calC^1$ or Lipschitz, respectively, with uniform control on the norms, 
and in the second setting, we assume that $\gamma_j$ is differentiable at the origin. Let us now 
describe these more carefully.  In the following, and throughout the rest of this section, we write 
\[
D_+(a) = \{(x,y,0): x^2 + y^2 \leq a^2,\ x \geq 0\}, \quad D(a) = \{(x,y,0): x^2 + y^2 \leq a^2\}
\]
for the half-disk or disk of radius $a$ in the vertical plane $\{z=0\}$. 

\medskip

\noindent {\bf Setting 1:} $Y_j$ is a sequence of incomplete Willmore surfaces, where each $Y_j$ is a horizontal graph
$z=u_j(x,y)$ over $D_+(3)$ with $u_j\in \calC^2$,  $||u_j||_{W^{2,2}}\le M<\infty$, $u_j(0,0)=0$ and $\del_y u_j(0,0)=\alpha>0$. 
We assume that $|\nabla u_j|_{\olg} \le \zeta\le 1/20$, and finally that $\calE(Y_j) \leq \mu<2\pi$ and $Y_j \to 
Y_* = D_+(3)$.  

\medskip 

\noindent {\bf Setting 2:}  $Y_j$ is a sequence of incomplete Willmore surfaces which are again horizontal graphs 
$z=u_j(x,y)$ over $D_+(3)$ with $u_j(0,0)=0$ and $u_j\in W^{2,2}$, $||u_j||_{W^{2,2}}\le M<\infty$, and $u_j\in \calC^2$ away 
from $\{x=0\}$. We assume that $(0,0)$ is a point of differentiability for $u_j$ 
and $\del_y u_j (0,0)=\alpha>0$. We also assume that $y \mapsto u_j(0,y)$ is Lipschitz with constant $\zeta \le 1/20$,
and furthermore, $|\nabla u_j|\le 2\zeta$ for $x>0$. 
Finally, suppose that $\calE(Y_j) \leq 2\pi$ and $Y_j \to Y_* = D_+(3)$.  

\medskip

Recalling that $\alpha$ is the jump in the derivative, choose any number $\beta$ with $0 < \beta \ll \alpha$. 
Consider the straight line $\ell_\beta = \{z = \beta y\}$ in the horizontal plane $\{x=0\}$. Since the curves 
$\gamma_j$ converge to a segment in the $y$-axis containing the subinterval $[-1,1]$, then for $j$ large, 
there must exist values  $-1 < y_j^- < 0 < y_j^+ < 1$ such that the two points $F_j^\pm=(0,y^\pm_j,u_j(y^\pm_j))$ both 
lie on the line $\ell_\beta$. We assume that $y_j^+$ is chosen as large as possible in the interval $(0,1)$, and similarly
for $y_j^-$. Since $\gamma_j={\rm Graph}(u_j|_{x=0})$ converges to the line $\ell_0 = \{z=0\}$, it is necessarily
the case that $|F_j^\pm| \to 0$. Let $R_{-\beta}$ denote the rotation of the $yz$-plane by the small negative 
angle which sends $\ell_\beta$ to $\ell_0$; thus $R_{-\beta}(F_j^\pm) = (\pm |F_j^\pm|,0)$. 

Suppose, to be definite, that $|F_j^+| \geq  |F_j^-|$. Dilating the entire surface by the factor $|F_j^+|^{-1}$ pushes 
the point $F_j^+$ to $(1,0)$. 
The key observation is that this dilation of $R_{-\beta} Y_j$ converges to a vertical plane (since it must be totally geodesic 
and graphical over $\{z=0\}$), and since this plane contains the two points $(0,0)$ and $(1,0)$, it {\it must} be 
$\{z=0, x \geq 0\}$. This holds even though, before dilating, the sequence $R_{-\beta} Y_j$ converges to the 
vertical plane $\{y = -\beta z, x \geq 0\}$. Denote this dilated, rotated surface by $\tilde{R}_{-\beta}(Y_j)$. Note also that 
our assumed $W^{2,2}$ bound implies that $\int_{Y_j}|\Abar|^2d\overline{\mu}\le M$, and hence given 
$\gamma>0$ and fixing $0<\beta_-<\beta_+\ll \alpha$, then for $j$ large enough there exists some 
$\beta \in (\beta_-,\beta_+)$ such that 
\[
\int_{\tilde{R}_{-\beta_j}(Y_j)\bigcap \{1/4\le x^2+y^2+z^2\le 9\}}|\Abar|^2d\overline{\mu}\le \gamma.
\]
\begin{remark}
\label{special.beta}
By this observation, we can pick a sequence $\beta_j, 0<\beta_-\le \beta_j<\beta_+\ll \alpha$  such that: 
\begin{equation}
\label{total.curv.zero}
\int_{\tilde{R}_{-\beta}(Y_j)\bigcap \{1/4\le x^2+y^2+z^2\le 9\}}|\Abar|^2d\overline{\mu}=o(1).
\end{equation}
We make this choice hereafter. 
\end{remark}
For simplicity, now reset the notation and write the rotated dilated surfaces as $Y_j$, with boundary curves
$\gamma_j$, graph functions $u_j$, etc. 
\begin{lemma}
\label{DeLell.appl}
Consider a sequence of incomplete Willmore surfaces $Y_j$ which are graphs $z=u_j(x,y)$ over $D_+(3)$ with $|\nabla u_j|\le 2\zeta$,
${\rm Lip}(u_j|_{x=0})\le \zeta$, $8\calE(Y_j)\le \pi$, $\int_{Y_j\cap \{1/4\le x^2+y^2+z^2\le 9\}}|\Abar|^2d\overline{\mu}\rightarrow 0$, 
$u_j(0,0)=0, u_j(0,1)=0$ and $u_j\rightarrow 0$, where the convergence is in $\calC^\infty$ away from $\{x=0\}$ and in $\calC^{0,\alpha}$ 
up to  $x=0$. Assume further that there is a jump in the first derivative at the origin:
\begin{equation}
 \label{new.jump}
\lim_{j\to \infty} \del_y u_j(0,0) - \del_y u_*(0,0) \ge \alpha-2\beta_j> \frac12 \alpha. 
 \end{equation}
Then there exist M\"obius transformations $\varphi_j$ and open sets $\calU_j\subset Y_j$ such that 
$\tilde{Y}_j=\varphi_j(\calU_j)$ are graphs $z=\tilde{u}_j(x,y)$ over $D_+(2)$ with $|\nabla \tilde{u}_j|\le 4\zeta$, 
$\tilde{u}_j(0,0)=0, \tilde{u}_j(0,1)=0$ and 
\begin{equation}
 \label{new.jump'}
\lim_{j\to \infty} \del_y \tilde{u}_j(0,0) - \del_y \tilde{u}_*(0,0) \ge \alpha-2\beta_j> \frac12 \alpha.
\end{equation}
Furthermore, there exist isothermal coordinate charts $(q_j, w_j)$ centered at the origin and covering the region
$\varphi_j(\calU_j)$ with $q_j = 0$ along $\gamma_j$, such that $C'^{-1}\le |\nabla q_j|_{\olg} \le C'$ for 
some constant $C' > 0$ which depends only on $\sup_j \calE(Y_j)$.  Finally, the conformal factor 
$\phi_j$ associated with the coordinates $q_j,w_j$ satisfies the estimates: 
 \begin{equation}
 \label{the.bounds}
||\phi_{j}||_{\calC^0(\varphi_j(\calU_j))} +||\phi_{j}||_{W^{1,2}(\varphi_j(\calU_j))}+||\phi_j||_{W^{2,1}(\varphi_j(\calU_j))}\leq
C \int_{\varphi_j (\calU_j)} | \Abring_{j}|^2\, d\overline{\mu}+o(1).
\end{equation}
\end{lemma}
\begin{remark}
\label{exists.extension}
Note, for future reference, that we actually prove that the surfaces $\varphi_j(\calU_j)$ are subregions 
of complete, smooth graphical surfaces ${Y}^\flat_j$ in $\RR^3$ which are reflection-symmetric 
across $\{x=0\}$. If $u_j^\flat(x,y)$ is the graph function of $Y^\flat_j$, then the (distorted) annular
regions $\{(x,y,u_j^\flat(x,y)), 2\le \sqrt{x^2+y^2}\le 4 \}$ of these larger surfaces are {\it not} Willmore with 
respect to the hyperbolic metric. On the other hand,  $u_j^\flat(x,y)=0$ for 
$\sqrt{x^2+y^2}\ge 5$; we denote this portion of $Y^\flat_j$ by $Y^\sharp_j$. 
The isothermal coordinates $(q_j,w_j)$ cover the entire surface $Y^\flat_j$, and the associated conformal factor $\phi_j$ 
satisfies \eqref{the.bounds} on all of $Y^\flat_j$ and $\phi_j\rightarrow 0$ as $\sqrt{x^2+y^2}\rightarrow \infty$.
In the first setting above, $\calE(Y^\flat_j)\to 0$. 
\end{remark}
\begin{remark} 
\label{translate.bounds}
The pointwise bound on $|\nabla q_j|_{\olg}$ follows from the $\calC^0$ bounds on $\phi_j$. 
Indeed, dropping the subscript $j$ momentarily, we have
\begin{equation}
\label{two.forms}
\olg =e^{2\phi}(dq^2+dw^2)=(1+(u_x)^2)dx^2+2u_x u_y dxdy+(1+(u_y)^2)dy^2,
 \end{equation}
so in particular
\[
\olg(\del_x, \del_x) + \olg(\del_y, \del_y) = 2+u_x^2 + u_y^2 = e^{2\phi}( |\del_x q|^2 + |\del_y q|^2 + |\del_x w|^2 + |\del_y w|^2).
\]
Using $|u_x|,|u_y|\le 1/10$ and $|dq|_{\overline{g}}=|dw|_{\overline{g}}$, $dq\perp dw$,  the equivalence of 
pointwise bounds on $\phi_j$ and $|\nabla q_j|_{\olg}$ follows directly. 
\end{remark}
\begin{proof} 
The main work is to establish the existence of the isothermal parametrization, so we concentrate on this; 
properties of the graphical representation are derived at the end.  

As described earlier, we apply a theorem of DeLellis and M\"uller \cite{DLM}, but there are a few technical points 
that must be addressed before we can do so.  First, the surfaces
in the statement of this theorem are local, so we must extend them to closed topological spheres. This
is done by first reflecting each $Y_j$ across the horizontal plane, then extending the resulting perturbed
disk to a surface which agrees with the vertical plane $\{z=0\}$ outside a large ball, then stereographically
projecting.  Next, to obtain a balanced configuration as described above (and in more detail below), we
can arrange for the perturbed $S^2$ to have two reflection symmetries immediately, then obtain the
third balancing condition by composing with an appropriate M\"obius transformation in $\RR^3$. 

The upshot is that we obtain isothermal coordinates $(q,w)$ which still detect the jump 
in the first derivative, and with $0 < C_1 \leq |\nabla q|, |\nabla w| \leq C_2$. 

\medskip

\noindent {\bf Reflection:}  We first reflect $Y_j$ across the horizontal plane to obtain a surface $Y_j'$ in 
$\RR^3$ invariant with respect to the vertical reflection $x \mapsto -x$. 
The doubled surface is graphical over $D(3)$, and has graph function $\tilde{u}_j \in W^{2,2}(D(3))$.  
This is straightforward to check using Lemma \ref{asympt.vertical}.
We change notation, denoting the doubled surface $Y'_j$ by $Y_j$ again. 

\medskip
 
\noindent {\bf Extension:}  We now claim that the doubled incomplete surface $Y_j$ can be extended 
to a complete surface $Y_j^{\mathrm{ext}}$ which is a graph over the entire vertical plane $\{z=0\}$ 
with graph function $u_j^{\mathrm{ext}}$  which vanishes when $x^2+y^2\ge 25$ and also satisfies
\begin{equation}
\label{energy.increase}
\calE(Y_j^{\mathrm{ext}}) \leq 2 \calE(Y_j)  + o(1).
\end{equation}
Notice that $Y_j^{\mathrm{ext}}$ is no longer Willmore in the transition annulus $4 \leq x^2 + y^2 \leq 9$. 

Since this construction is a bit lengthy, we defer it to \S 4 below, so let us grant it for the time being.

\medskip

\noindent {\bf Mollification:}  Since we apply the uniformization theorem later, it is convenient 
mollify these surfaces. We check convergence properties of the uniformization map for the 
mollified surfaces momentarily. The mollification is standard:
Choose $\psi\in \calC^\infty_c(\RR^2)$ with $\int\psi =1$ and set $\psi_\e(x,y):=\e^{-2}\psi(x/\e, y/\e)$.  
Then $Y_{j,\e}^{\mathrm{ext}}$ is the graph of 
$$
u^{\mathrm{ext}}_{j,\e}(x,y)=u^{\mathrm{ext}}_j*\psi_{\e} (x,y). 
$$
It is standard that for each $\e \in (0,1]$, $u^{\mathrm{ext}}_{j,\e}\in \calC^\infty$ and $||u^{\mathrm{ext}}_{j,\e}||_{W^{2,2}}
\to ||u^{\mathrm{ext}}_{j}||_{W^{2,2}}$.  Furthermore, given a priori $\calC^1$ or Lipschitz bounds on $u^{\mathrm{ext}}_j$, there 
are uniform $\calC^1$ bounds (independent of $\e$ and $j$) on $u^{\mathrm{ext}}_{j,\e}$.  In particular, 
\begin{equation}
\label{energy.convergence}
\int_{Y^{\mathrm{ext}}_{j,\e}}|\Abar_{j,\e}|^2d\bar{\mu} \to \int_{Y^{\mathrm{ext}}_j}|\Abar_{j}|^2d\bar{\mu}
\end{equation}
as $\e \to 0$. 

\medskip

\noindent {\bf Patching into a sphere with symmetries:} 
Let $I$ be the M\"obius transformation of $\RR^3$ which maps the plane $\{z=0\}$ 
to the sphere $S_1(0)$, normalized by requiring that $I( (0,0,0)) = (0,0,-1) := S$ (the south pole), $I(\infty) = (0,0,1) := N$ 
(the north pole), and so that $I$ carries the $y$-axis to the great circle $C := \{z=0\} \cap S_1(0)$ minus $N$. 
This map is fully determined by requiring that the image of the disc $D(2)$ equals the spherical 
cap in $S_1(0)$ of radius $1/10$ centered at $S$.  

Now let $Y'_{j,\e}:=I (Y_{j,\e} \cup \{\infty\})$; this is a slightly distorted sphere, where the distortion is localized near $S$.
This surface has one reflection symmetry, across the plane $x=0$, corresponding to the original vertical reflection
symmetry. By construction, $Y_{j,\e}'$ coincides with standard unit sphere in a neighbourhood of the closed northern 
hemisphere $S_1(0) \cap \{z \geq 0\}$, so we can form a new surface $\hat{Y}_{j,\e}$ by discarding this northern 
hemisphere and replacing it with a reflection of the southern hemisphere of $Y_{j,\e}'$. The surface we obtain this
way is smooth near its intersection with the horizontal plane $\{z=0\}$, and has two reflection symmetries, one across the 
plane $\{z=0\}$ and the other across $\{x=0\}$. 

Now, consider a uniformizing map $\psi_{j,\e}$ from the round sphere to $(\hat{Y}_{j,\e},{\olg}_{j,\e})$,
where $\olg_{j,\e}$ is the metric on $\hat{Y}_{j,\e}$ induced from $\RR^3$.  Because of the symmetries of 
$\hat{Y}_{j,\e}$, the map $\psi_{j,\e}$ can be chosen to be reflection-symmetric across the $xy$- or $yz$-planes. 
We can also assume that the conformal maps $\psi_{j,\e}$ converge in $W^{1,\infty}$ to a conformal map $\psi_j$ 
as $j\rightarrow \infty$.  This is done using an inversion, coupled with the graphicality property and \cite{MS}.
Indeed, first consider the inversion $\tilde{I}$ of $\RR^3$ which sends $(1, 0, 0)$ to infinity and fixes $(-1, 0,0)$. 
It follows (by the same argument as in the graphicality discussion below) that $\tilde{I}(\hat{Y}_{j,\e})$ is a graph 
over the plane $\{x=0\}$.  We thus obtain a graph function $x=f_{j,\e}(y,z)$ such that $f_{j,\e}=0$ for $y^2+z^2\ge M$,
where $M$ can be chosen independent of $\e$.  By the first paragraph in the proof of Theorem 5.2 in \cite{MS},
there exists a 1-parameter family of smooth conformal parametrizations $\hat{\psi}_{j,\e}:\mathbb{R}^2\rightarrow 
\tilde{I}(\hat{Y}_{j,\e})$ with $\hat{\psi}_{j,\e}\rightarrow 0$ at $\infty$, such that $\hat{\psi}_{j,\e}$ 
converges in $W^{1,\infty}$ to a parametrization $\hat{\psi}_j:\RR^2\rightarrow \tilde{I}(\hat{Y}_j)$.
Since the conformal factor $\hat{\phi}_{j,\e}$ associated to $\hat{\psi}_{j,\e}(y,z)$ is harmonic with respect 
to the flat metric on $y^2+z^2\ge M$, standard asymptotics results for harmonic functions on exterior 
domains give that $|\del^2 \hat{\phi}_{j,\e}|=o((x^2+y^2)^{-1})$, 
uniformly in $\e$ for $j$ fixed. This implies that the conformal maps $\psi_{j,\e}:=\tilde{I}^{-1}\circ 
\hat{\psi}_{j,\e}\circ \tilde{I}:S^2\rightarrow \hat{Y}_{j,\e}$ converge in $W^{1,\infty}$ to a conformal map 
$\psi_j:S^2\rightarrow \hat{Y}_j$. We consider these $\psi_{j,\e}$ hereafter. 
  
\medskip

\noindent {\bf Balancing the total curvature:} This is the key step in the derivation of our estimates. 
 As already mentioned,  \cite{DLM} requires that we find a conformal map $\psi_{j,\e}: S^2 \to \hat{Y}_{j,\e}$ 
which satisfies three separate balancing conditions: if $H_{\pm, b}$, $b=x,y,z$, denote the three pairs of hemispheres in 
$S^2$ centered along the coordinate axes with the same labels, then we demand that
\begin{equation}
\int_{\psi_{j,\epsilon} (H_{+,b})}  |\overline{A}_{j,\e}|^2 \, d\bar{\mu} = \int_{\psi_{j,\epsilon} (H_{-,b})}  |\overline{A}_{j,\e}|^2 \, d\bar{\mu},
\label{balanceb}
\end{equation}
for all three choices of $b$. Our conformal map $\psi_{j,\e}: S^2 \to \hat{Y}_{j,\e}$ clearly respects the 
two reflection symmetries of this target surface, and for this map, \eqref{balanceb} is satisfied for $b=x$ and $z$.

To obtain the third balancing condition, we modify $\psi_{j,\e}$ by composing it with a M\"obius transformation
$M_t$ which is a hyperbolic dilation with source $(0,-1,0)$ and sink $(0,1,0)$. Notice that $\psi_{j,\e}$ fixes 
these two points already. Each $M_t$ preserves the unit sphere $S_1(0)$. 
Now consider the family of surfaces $M_t( \hat{Y}_{j,\e})$. These all have the two original reflection 
symmetries simply because $M_t$ respects those reflections.  We claim that for each $j,\e$, there exists 
a unique $t_{j,\e}$ such that 
\[
\int_{M_{t_{j,\e}}\circ \psi_{j,\e} (H_{+,y}) } |\Abar_{j,\e}|^2 \, d\bar{\mu} = \int_{ M_{t_{j,\e}}\circ \psi_{j,\e} (H_{-,y})} |\Abar_{j,\e}|^2 \, d\bar{\mu}.
\]

To prove this, first note that if $\Sigma$ is a smooth closed surface in $\RR^3$ diffeomorphic to the sphere, then
$|\Abar_\Sigma|^2 = 2 |\Abring_\Sigma|^2 + 2K_\Sigma$, where $K_\Sigma$ is the Gauss curvature of $\Sigma$, hence 
\[
\int_\Sigma |\Abar_\Sigma|^2 \, d\bar{\mu} = 2\int_\Sigma |\Abring_\Sigma|^2 \, d\bar{\mu} + 8\pi.
\]
The two terms on the right are conformally invariant, and hence preserved if we apply any one of the maps $M_t$ to $\Sigma$;
therefore so is the left side.  We also remark here that by construction, $8\pi\le \int_{\hat{Y}_{j,\e}}|\Abar_{j,\e}|^2 \, 
d\bar{\mu}\le 10\pi$ for $j$ large. Now, $\hat{Y}_{j,\e}$ agrees with the standard round sphere in a small neighbourhood
around the two points $W = (0,-1,0)$ and $E = (0,1,0)$.  Recall also that $\psi_{j,\e}:S^2\rightarrow \hat{Y}_{j,\e}$ maps 
the points $ {\bf W}=(0,-1,0)$, ${\bf E}=(0,1,0)$ in $S^2$ to $W$, $E$ in $\hat{Y}_{j,\e}$. Choose a small disk centered at 
$\bf W$; its image under $\psi_{j,\e}$ will then be a small cap around $W$. Now, the image of this cap 
under $M_t$ with $t \gg \infty$ is a large spherical cap which covers almost the entire sphere except a small neighbourhood
of $E$, hence the integral of $|\Abar|^2$ over this region is very close to $8\pi$. 
Analogously, at $t\gg-\infty$ its image is a tiny spherical cap centered at $W$, hence the integral 
of $|\Abar|^2$ over this region is very close to $0$. Since the total energy of $\hat{Y}_{j,\e}$ is just slightly larger than 
$8\pi$, this gives the existence of the value $t_{j,\e}$, as claimed.  Since $\int_K |\Abar_{j,\e}|^2 \to \int_K |\Abar_j|^2$ 
and the conformal maps $\psi_{j,\e}\rightarrow \psi_j$ uniformly in all of $S^2$ as $\e \to 0$ for each fixed $j$, 
where $K$ is any fixed closed subset, this argument shows that there is a 
bound $|t_{j,\e}| \leq T_j$ which is uniform in $\e$. 

The preferred conformal transformation is now given by
\[
\Psi_{j,\e} = M_{t_{j,\e}} \circ \psi_{j,\e}:  S_1(0) = S^2 \longrightarrow M_{t_{j,\e}}(\hat{Y}_{j,\e}). 
\]
If $\olg_0$ is the standard round metric on $S_1(0)$ and $\hat{g}_{j,\e}$ is the metric on $M_{t_{j,\e}}(\hat{Y_{j,\e}})$ 
induced from the Euclidean metric in $\mathbb{R}^3$, then define $\phi_{j,\e}$ by 
\[
\Psi_j^* \hat{g}_{j,\e} = e^{2\phi_{j,\e}} \olg_0.
\]
Proposition 3.2 and Theorem 3.3  in \cite{DLM} and the $W^{2,1}$ estimates in their proof now give that
\begin{equation}
||\phi_{j,\e}||_{\calC^0} +||\phi_{j,\e}||_{W^{1,2}}+||\phi_{j,\e}||_{W^{2,1}}\leq C \int_{M_{t_{j,\e}} (\hat{Y}_{j,\e})} | \Abring_{j,\e}|^2\, d\bar{\mu}.
\label{Hardyest}
\end{equation}
As a brief hint of the idea of the proof of this fact, $\phi_{j,\e}$ is a solution of the semilinear elliptic PDE,
$\Delta_{g_0} \phi_{j,\e} = 1 - \hat{K}_{j,\e} e^{2\phi_{j,\e}} $, where $\hat{K}_{j,\e}$ is the Gauss curvature function on 
$M_{t_{j,\e}}(\hat{Y}_{j,\e})$. The main term $\hat{K}_{j,\e} e^{2\phi_{j,\e}}$ on the right has a `determinant structure', since it 
can be expressed via the pullback of the area form on $S^2$ by the Gauss map. After appropriately modified
stereographic projections (localized to be trivial in certain regions of the sphere)  this allows one to conclude that 
the right hand side lies in the Hardy space ${\mathcal H}^1(\RR^2)$, and from there the estimates 
follows from some important and well-known theorems in harmonic analysis. We refer to \cite{DLM} 
and for further details.

We now pass to the limit as $\e \to 0$. Because $t_{j,\e}$ is bounded uniformly in $\e$ for each $j$ we can pass to 
a subsequence (in $\e$) and assume that $t_{j,\e}\to t_j$; with no loss of generality, and possibly taking
a reflection, we assume that $t_j\ge 0$ for all $j$.  Following the argument in M\"uller-Sverak \cite{MS}, we 
obtain limiting functions $\psi_{j,\e} \to \psi_j$ and $\phi_{j,\e} \to \phi_j$, where 
\begin{equation}
\label{prethe.bounds}
||\phi_j||_{\calC^0} +||\phi_{j}||_{W^{1,2}}+||\phi_j||_{W^{2,1}}\leq C \int_{M_{t_{j}} (\hat{Y}_{j,\e})} | \Abring_{j}|^2\, d\bar{\mu}\le 
4C\calE(Y_j)+o(1).
\end{equation}
The second inequality here follows from (\ref{energy.increase}), (\ref{energy.convergence}) and 
the assumption that the total curvature in an annular region converges to zero. 
\medskip

\noindent {\bf Undoing the stereographic projection:}  For brevity, set $\overline{Y}_j = M_{t_j}(\hat{Y}_j)$. 
The points $P_j$ and $Q_j$ which correspond to the points of intersection $(0,0,0)$ and $(0,1,0)$ of the boundary 
curve and the $y$-axis in the original (dilated and rotated) surface $Y_j$ correspond under $M_{t_j} \circ \psi_j$ 
to new points, which we still label as $P_j$ and $Q_j$, on $\overline{Y}_j \bigcap \{x=0\}$. These 
lie on the circle $S_1(0)\bigcap \{x=0\}$. 

We have denoted by ${\bf E}$ and ${\bf W}$ the points $(0,1,0)$ and $(0,-1,0)$ in $S^2$, respectively; the corresponding points 
$(0,\pm 1, 0)$ in $\overline{Y}_j$ have been labelled $E$ and $W$. 
Let $B$  be the M\"obius transformation of $\RR^3$ which induces the stereographic projection 
from $W$ onto the plane $\{y=1\}$. Let $B^\sharp$ be the stereographic projection from $S^2 \setminus {\bf W}$ to $\RR^2$, 
which sends ${\bf E}$ to $0\in \mathbb{R}^2$ and $\bf W$ to $\infty$. 

If $\tilde{Y}_{j}:= B(\overline{Y}_{j})$, then define
$$
\tilde{\Psi}_j : \RR^2\to \tilde{Y}_{j}, \qquad \tilde{\Psi}_{j}:=  B\circ \Psi_{j}\circ (B^\sharp)^{-1}.
$$
This map is conformal, and determines the functions $q_{j}, w_{j}$ as the 
push-forwards of the flat coordinates $q,w$ on $\RR^2$, so that $\tilde{\Psi}_{j}^*(\olg)= 
e^{2\tilde{\phi}_{j}}(dq_{j}^2+dw_{j}^2)$.  The points $P_j, Q_j$ are mapped to points $\tilde{P}_j, \tilde{Q}_j$ 
on the line $\{y=1,x=0\}$

If $B(\tilde{P}_j, |\tilde{P}_j\tilde{Q}_j|)$ denotes the ball in $\RR^3$ centered at $\tilde{P}_j$ and with radius
$|\tilde{P}_j\tilde{Q}_j|$, then on $\tilde{Y}_{j}\bigcap B(\tilde{P}_j, |\tilde{P}_j,\tilde{Q}_j|)$, the conformal factor $\tilde{\phi}_{j}$ is
obtained by adding to $\phi_j$ a smooth function $w_j$ which is {\it a priori} bounded since it the conformal 
factor for a stereographic projection restricted to the domain $\tilde{Y}_{j}\cap B(\tilde{P}_j, |\tilde{P}_j,\tilde{Q}_j|)$ 
which is uniformly bounded away from the point that is mapped to infinity.  It is not hard to see that since 
$\tilde{Y}_{j}\cap B(\tilde{P}_j, |\tilde{P}_j,\tilde{Q}_j|)$ converges to a vertical half-plane 
through $\tilde{P}_j, \tilde{Q}_j$, $|w_j|_{\calC^2}=o(1)$. Thus we obtain isothermal coordinates $(q_j, w_j)$ on 
$\tilde{Y}_{j}$ and a conformal factor $e^{\phi_j}$ which satisfies (\ref{the.bounds}). 

\medskip

\noindent {\bf The final dilation and  the jump in the derivative:} 
Set $\tilde{Y}_{j}:= B^\sharp \circ M_{t_{j}} \circ I (Y_{j}^{\mathrm{ext}} ) \bigcap \{x\ge 0\}$; this is a surface with boundary 
$\tilde{\gamma}_{j} = B^\sharp \circ M_{t_{j}} (Y_{j}') \bigcap \{x=0\}$. Write $P_j:=\widetilde{\Psi}_{j}  (0,0)$ 
and $Q_{j}:=\widetilde{\Psi}_{j}(1,0)$; these both lie on the line $\{y=1,x=0\}$, and since all maps here are conformal, 
$\tilde{\gamma}_{j}$ makes an angle bigger than $\alpha/2$ with this line  at $P_{j}$. 
Translate and rotate so that $P_{j}$ is the origin and $Q_{j} = ( 0,d_j, 0)$. 

Let $F_{j}$ be the Euclidean dilation from the origin by the factor $d_{j}^{-1}$, so that $F_{j}(Q_{j}) = (0,1,0)$.
Note that $F_{j} (\tilde{\gamma}_{j})$ still makes an angle bigger than $\alpha/2$ with the $y$-axis at 
$(0,0,0)$. Pre- and postcomposing $\tilde{\Psi}_{j}$ by $F_{j}$ gives a conformal map 
$$
\widehat{\Psi}_{j} := F_{j}\circ \widetilde{\Psi}_{j} \circ F_{j}^{-1}:\RR^2\longrightarrow  F_j(\tilde{Y}_{j}),
$$ 
which leaves the conformal factor $\tilde{\phi}_{j}$ unchanged. 

Since $\tilde{Y}_{j}$ converge to a vertical plane, the curves $F_{j}(\tilde{\gamma}_j)$ converge 
in $\calC^\alpha$ to the $y$-axis as $j\to\infty$. The required hyperbolic isometries $\varphi_j$ are 
then just $\widehat{\Psi}_j$, and the domains $\calU_j$ are the preimages of the graph of the unit disc under these maps. 

\medskip

\noindent {\bf Graphicality:} 
We return finally to the claim that the surfaces $\widetilde{Y}_j$ remain graphical. Recall first that $Y_j$ is graphical 
over the disc $\{x^2+y^2 \le 9, z=0\}$, and that $|\nabla u_j|\le 2\zeta$.  Consider the family of straight lines 
$z \mapsto (x_0, y_0, z)$ parallel to the $z$-axis. Each of these meet $\widetilde{Y}_j$ at an angle $\Omega(x_0,y_0)$ 
which satisfies $|\Omega - \pi/2| < \arctan (2\zeta)$. Under the M\"obius transformation $I$, this family is transformed to a family 
of circles $\calC_N$, each passing through $(0,0,1)$ and meeting $S_1(0)$ orthogonally.  By conformality, 
if $C$ is one of these circles which intersects $I(\widetilde{Y}_j)$ at a point $Q$ with angle $\Phi(Q)$, then 
$|\Phi(Q)-\pi/2|\le \arctan (2\zeta)$. 

Now consider the set of circles $\calC_E$ passing through $(0,1,0)$ and intersecting $S_1(0)$ 
orthogonally.  For any point $Q\in I(\calD_j)$, consider the circles $C_N(Q) \in \calC_N$ and
$C_E(Q) \in \calC_E$ which pass through $Q$.  Since $\mbox{dist}( Q, S) \leq \arctan(1/10)$, 
the angle between $C_N(Q)$ and $C_E(Q)$, is at most $\arctan(2\zeta)$. Hence for every  $Q\in I({\cal D}_j)$, 
the angle $\Phi'(Q)$ between $C_E(Q)$ and the surface $I(\widetilde{Y}_j)$ satisfies $|\Phi'(Q)-\pi/2|\le 2 \arctan(2\zeta)$. 

Finally, note that the dilations $M_t$ preserve the family $\calC_E$. By conformality, for each point 
$Q\in M_{t_j}(I(\widetilde{Y}_j))$, the surface $M_{t_j}(I(\widetilde{Y}_j))$ makes an angle $\Phi'(Q)$ with $C_E(Q)$, 
where $|\Phi'(Q)-\pi/2|\le 2\arctan(2\zeta)$.  Recall that $B^\sharp$ maps each $C_E(Q)$ to a line parallel to the $z$-axis. 
By conformality again, these lines make an angle $\Phi'(Q)$ with $B^\sharp \circ M_{t_j}(I(\widetilde{Y}_j))$ at the 
point of intersection $Q$, where $|\Phi'(Q)-\pi/2|\le 2\arctan(2\zeta)$.  But this means precisely that 
$B^\sharp\circ M_{t_j}(I(\widetilde{Y}_j))$ is a graph over a disk of some 
graph function $\tilde{u}_j$ satisfying $|\nabla \tilde{u}_j|\le 4\zeta$.
\end{proof}

\subsection*{Construction of the extension.}
\label{extension.section}
We now prove the fact claimed in the proof of Lemma~\ref{DeLell.appl}
that the reflected surface $Y'_j$ 
can be extended to a graph over the entire plane $\{z=0\}$ in such a way that the increase of energy is controlled.
This is straightroward using mollification. The point is that each of our surfaces is graphical with bounded tilt, 
so the total curvature is equivalent to the $L^2$-norm of the Hessian of its graph function. In particular, 
if $Y={\rm Graph}(u)$ for $u\in \calC^2(D')$, $D'=\{1/4\le x^2+y^2\le 9\}$, with $|\nabla u|\le 2\zeta$, then 
\begin{equation}
\label{two.forms}
\frac{1}{(1+4\zeta^2)}\int_{D'}|\del^2 u|^2dxdy\le \int_{Y'}|\Abar|^2d\overline{\mu}\le (1+4\zeta^2)\int_{D'}|\del^2 u|^2dxdy.
\end{equation}

\begin{lemma}\label{extension}
Let $u$ be a $W^{2,2}$ function defined on the half-disc $D_+(3):=\{\sqrt{x^2+y^2}\le 3, x>0 \}$.  If $Y={\rm Graph}(u)$ 
then write 
\[
\int_{Y} |\Abring|^2d\bar{\mu}:= \calE, \quad 
\int_{Y\cap 1/2\le \sqrt{x^2+y^2}\le 3} |\Abar|^2d\bar{\mu}:= \calE'
\]
and assume that $|\nabla u|\le 1$, and in addition
\begin{enumerate}  
\item There exist $\epsilon, \delta >0$ such that $|\nabla u(P)|\le \delta$ for all $P\in D'\bigcap \{x\ge\epsilon\}$;
\item  For any $P\in \overline{D'}\cap \{x=0\}$, and any sequence $P_j\in D'$ with $P_j\rightarrow P$ we have 
$\lim_{j\to \infty}\partial_x u(P_j)=0$.
\end{enumerate}
Let $U$ be the even extension of $u$ to $D=\{\sqrt{x^2+y^2}\le 3\}$.  Then there exists a function $\overline{u}$ 
such that $\overline{u}=U$ on $\{\sqrt{x^2+y^2}\le 1\}$, $\overline{u}=0$ on $\{\sqrt{x^2+y^2}\ge 5\}$ 
and if we let $\overline{Y}:={\rm Graph}(\overline{u})$ then 
$\int_{\overline{Y}}|\Abring_{\overline{Y}}|^2 \le 2 \calE +1000(\delta+\epsilon)+10\calE'$. 
\end{lemma}
By Remark~\ref{special.beta} and the fact that the $Y_j$ converge locally in $\calC^\infty$ to a vertical half-plane
away from $\{x=0\}$, this Lemma then implies the claim on extension from above. 
\begin{proof}  First note that if $u\in W^{2,2}$, then using the fact that $\partial_xu=0$ on $\{x=0\}$, we have 
$U \in W^{2,2}$. Furthermore, using the formula for the second fundamental form of a graph $z = U(x,y)$, 
we have 
\[
\int_{Y'}|\Abring|^2 d\overline{\mu}= 2\calE.
\]

To construct the extension, fix a smooth cutoff function $\chi(x,y)\in \calC^\infty_0(B_1(0))$ with 
$\int_{\mathbb{R}^2} \chi dxdy=1$, and such that $|\del \chi|\le 10, |\del^2\chi|\le 100$. Given any $\rho>0$ we let 
$\chi_\rho:=\rho^{-2}\chi(\frac{x}{\rho},\frac{y}{\rho})$. We work in polar coordinates $r:=\sqrt{x^2+y^2}$, 
$\theta:=\arctan (y/x)$. 

Define a function $u^\sharp(r,\theta)$ which equals $u(r,\theta)$ for $r\le 5/2$, and which vanishes for $r>5/2$. 
In addition, let $\psi(r)$ be a $\calC^\infty$ function which vanishes when $r \leq 1$, equals $2$ for $r \geq 3$,
is strictly monotone increasing in the interval $[1,3]$ and satisfies $|\psi'(r)|\le 10, |\psi''(r)|\le 100$.
Then define the function
\[
\overline{u}(r,\theta):=(u^\sharp*\chi_{\psi(r)})(r,\theta),
\]
where $\chi_0$ is understood as the $\delta$ function. It is straightforward that $\overline{u}$ is 
$\calC^2$ away from $\{x=0\}$, and it is also obvious that $\overline{u}(r,\theta)=0$ for $r\ge 5$ and that 
 $|\nabla\overline{u}|\le 1$ throughout $\RR^2$. What remains is to show that the surface $Y={\rm Graph}(\overline{u})$
satisfies the claims of our Lemma. 

\medskip
  
To do this, we recall some facts about the Hardy-Littlewood 
maximal functions. For $\chi \in \calC^\infty_0$ and $f\in L_{loc}^1(\mathbb{R}^2)$, define
\[
M(f)(x):=\sup_{\rho>0}|(f*\chi_\rho)(x)|.
\]
Then (for an appropriate choice of cutoff function $\chi$), 
\begin{equation} \label{hardybound}
||M(f)||_{L^2}\le 10||f||_{L^2}.
\end{equation}

Using (\ref{hardybound}) and (\ref{two.forms}) we derive:
\[
\int_{Y\bigcap \{1\le r\le 2\}}|\Abar|^2d\overline{\mu}\le 20\calE'.
\]
This implies immediately that
\[
\int_{Y\bigcap \{ r\le 2\}}|\Abring|^2d\overline{\mu}\le 2{\cal E}+20\calE'.
\]
Thus matters are reduced to estimating $\int_{Y_j\bigcap \{ r\ge 2\}}|\Abring|^2d\overline{\mu}$.
Given (\ref{two.forms}), it suffices to control:

\[
\int_{\{r\ge 2\}} |\del^2 \overline{u}|^2dxdy.
\]
We use the formula $\del^2 \overline{u}(P)=[\del^2(\chi_{\psi(r(P))})*u^\sharp](P)$. 
Using the pointwise bounds on $\partial^j\chi_r$, $\partial^j\psi$, $j=0,1,2$ and on $|u|$, we 
directly derive that: 
\[
\int_{Y\bigcap \{2\le r\le 5\}}|\Abar|^2d\overline{\mu}\le 1000\int_{\RR^2\bigcap \{2\le r\le 5\}}|\overline{u}|^2dxdy\le 1000(\e+\delta)
\]
Finally, using the definition of $u^\sharp$ we deduce that $\int_{Y\bigcap \{5\le r\}}|\Abar|^2d\overline{\mu}=0$. 
\end{proof}

\section{The key estimates: Small weighted energy in a half-ball controls $\calC^1$ regularity}
\label{energy.controls}

We now prove Propositions \ref{theclaim} and \ref{ereg.prp2}.
In both cases, consider the sequence of (incomplete, graphical) Willmore surfaces 
furnished by Lemma \ref{DeLell.appl}, but dilate these further so that they are all graphical over $D_+(3)$.
We denote by $\hat{Y}_j$ the image of the original surfaces under the M\"obius transformation in the previous subsection. 

Recall also the connection between the bounds on the conformal factor $\phi_j$ and on $|dq_j|_{\olg}$ and $|dw_j|_{\olg}$,
as explained in Remark 4.3.  
The bounds $||\phi_j||_{\calC^0} \leq C_1$ and $\calE(Y_j)\le \e'(\zeta) \ll 1$ imply that 
the image of the entire rectangle $0\le q_j\le 1, |w_j|\le 1$ lies in $Y_j$. Furthermore, since $Y^\flat_j$ converges to a
vertical plane $Y_*$ (see  Remark \ref{exists.extension}), and since $q_j$ has gradient bounded above and below and vanishes 
along $\{x=0\}$, and $\Delta_{\olg_j} q_j = 0$, there is a subsequence of the $q_j$ converging to a harmonic function
of linear growth which vanishes at $x=0$. The only possible limit is $\lambda x$ for some $\lambda  \in [C^{-1}, C]$.
As before, the convergence is $\calC^\infty$ away from $\{x=0\}$.  

Next, using the $\calC^0$ and $W^{1,2}$ bounds for $\phi_j$ in \eqref{prethe.bounds}, 
we may replace coordinate derivatives by covariant derivatives in these bounds, at worst
only increasing the constant. 

\medskip

We can now prove the main analytic estimate, which shows that the energy of $Y_j$ 
controls the jump in the first derivative of the boundary curve of $Y_j$ at the origin.   
  
In the following, we often write $\del_1$ and $\del_2$ for the coordinate vector fields $\del_{q_j}$, $\del_{w_j}$ on $Y_j$,
and also set $\Abar_j(\del_1, \del_2)  = (\Abar_j)_{12}$, or simply $\Abar_{12}$. 
Since the coordinate function $z$ equals $u_j$ on $Y_j$, we have
\begin{equation}
 \label{nabla.ext.curv}
(\bar{A}_j)_{12} \nabla_\nu z =\nabla_{12} u_j =\del_{12} u_j-\del_1u_j\, \del_2\phi_j-\del_1\phi_j\, \del_2u_j.
\end{equation}
The two equalities are just specializations of basic definitions to this situation. Now drop the subscript $j$ for simplicity.   
Multiply the equation by $e^{-\phi}$. Noting that $e^{-\phi} (\del_{12} u - \del_1 \phi \del_2 u) = \del_1( e^{-\phi} \del_2 u)$, 
then integrating along the line segment $0 \leq q \leq 1$, $w = 0$ gives 
\begin{equation}
\label{after.integration}
(e^{-\phi}\del_2 u )(1,0)- (e^{-\phi}\del_2 u)(0,0) =  \int_{(0,0)}^{(1,0)}  e^{-\phi} \bar{A}_{12} \nabla_\nu u \, dq +
\int_{(0,0)}^{(1,0)}e^{-\phi}\del_2\phi\del_1u\, dq.
\end{equation}
  
We now state our main estimate. 
\begin{proposition}
\label{two.bounds}
Set $\calE_j:=\calE(Y_j)$ and ${\cal E}_{j,p}={\calE}_p(Y_j)$. Then there exist constants $C,C'$ such that
\begin{equation}
\label{big}
\int_{(0,0)}^{(1,0)} \left|(\bar{A}_j)_{12}e^{-\phi_j}\right|\, dq_j  +\int_{(0,0)}^{(1,0)}e^{-\phi_j}\partial_2\phi_j\partial_1z\, dq_j
\le C\sqrt{\calE_{j,p}}  + C'\sqrt{\calE_j} +o(1). 
\end{equation}
\end{proposition}  

Before proving this, let us explain how it leads to a contradiction, thus establishing Propositions \ref{theclaim} and \ref{ereg.prp2}. 
In view of the $\calC^\infty$ convergence of $u_j \to u_*$ and $(q_j,w_j) \to (q_*,w_*)$ (both for $x > 0$), 
\begin{equation}
\label{summand1}
|e^{-\phi_j}\partial_{w_j} u_j(1,0)-e^{-\phi_*}\partial_{w_*} u_*(1,0)|\le \alpha/10 
\end{equation}
for $j$ sufficiently large. Furthermore, 
\begin{equation}
\label{int1}
e^{-\phi_j}\partial_{w_j} u_j(1,0)-e^{-\phi_j}\partial_{w_j} u_j(0,0) = \int_0^1 \partial_{1}(e^{-\phi_j}\partial_2 u_j)\, dq_j,
\end{equation}
and since $Y_*$ lies in the plane $\{z=0\}$, we also have
\begin{equation}
\label{int2}
e^{-\phi_*}\partial_{w_*} u_*(1,0)-e^{-\phi_*}\partial_{w_*} u_*(0,0)=0.
\end{equation} 
Proposition~\ref{two.bounds} then yields that for large enough $j$
\begin{equation} \label{control}
\int_{(0,0)}^{(1,0)} \del_{q_j}(e^{-\phi_j}\partial_{w_j} u_j)\, dq_j\le \alpha/10.
\end{equation}
Combining these facts along with $|e^{-\phi_j} \partial_{w_j} u_j-\partial_y u_j|\le \frac{1}{10}|\partial_y u_j|$,
which holds since  $|\nabla u_j|<2\zeta\le 1/10$ and $|\del_{w_j} - \del_y|_{\olg}$ is also small, 
we conclude that 
\begin{equation}
\label{contradiction}
\left|\partial_yu_j(0,0)-\partial_yu_*(0,0)\right|\le \alpha/3
\end{equation}
when $j$ is large, which contradicts \eqref{new.jump}.

\subsection{Regularity from the interior: the two line integrals}
\label{integral.2ff}
Proposition~\ref{two.bounds} is a consequence of the following two results: 
\begin{proposition}
\label{decay} 
With all notation as above, suppose that $||\phi_j||_{\calC^0(Y_j)}\le K$ for all $j$. 
Then there exists a constant $C(K)>0$ such that for each point $P\in Y_j$ 
with $q_j(P)\in[0,1]$ and $w_j(P)=0$ we have 
\begin{equation}
\label{eqndecay}
|(\Abring_j)_{12}|(P)\le C(K)\frac{\sqrt{\int_{B^{2}(P)}|\Abring_j|^2f_j^{2p}d\bar\mu}}{U_j(q_j)},
\end{equation}
where $B^2(P)$ is the (intrinsic) ball of radius 2 centered at $P$ and the functions $U_j(q_j)$ satisfy 
 $\int_0^1\frac{dq_j}{U_j(q_j)}\le M'<\infty$
for some uniform constant $M'$.
\end{proposition}
  
\begin{proposition} \label{bound.phi.conseq}
For some constants $C,C'$ independent of $j$, we have
\begin{equation} \label{second.bound}
\int_{(0,0)}^{(1,0)} e^{-\phi_j}\partial_2\phi_j\partial_1u_j \, dq_j \le C  \int_{Y_j} |\Abring_j|^2\, d\bar\mu+ 
C'\sqrt{\int_{Y_j} |\Abring_j|^2\, d\bar\mu}+o(1);
\end{equation}
 \end{proposition}

These are proved in the remaining subsections of \S 5. 

\begin{remark}
\label{bd.comparison}
In our setting, the $\epsilon$-regularity result \cite[Theorem I.5]{riv1} applied to intrinsic discs of radius $1$ 
in $Y_j$ (with respect to $g_j$) yields that $|\Aring_j|\le C\sqrt{{\cal E}_j}$, which implies that $x|\Abring_j|  \le C\sqrt{{\cal E}_j}$. 
Using that $0 < C_1 \leq q_j/x < C_2$, which we have already noted follows from the upper and lower bounds on $||\phi_h||_{{\cal C}^0}$, 
we see that $q_j|\Abring_j|  \le C'\sqrt{{\cal E}_j}$. Therefore, Proposition~\ref{decay} actually shows that assuming
bounded {\it weighted} energy yields a stronger pointwise decay estimate for $|\Abring_j|$.
\end{remark}
\subsection{Proof of the Proposition \ref{decay}:}
The argument relies on obtaining pointwise control on $|\Abring|$ at each point on the segment $\{0\le q_j\le 1, w_j=0\}$ 
using the weighted energy of $Y_j$ on a ball of (hyperbolic) radius 1 around that point. 

To this end, we use a well-known realization of Willmore surfaces as harmonic maps into the $(3+1)$-dimensional 
deSitter space $(dS_{1,3},h)$, which we regard as a hypersurface in the $(4+1)$-dimensional Minkowski space 
$\mathbb{R}^{1,4}$. This is useful since the norm of this map, $|d\Phi|^2$, is precisely equal to $|\Abring|^2$. 

\bigskip
 
\noindent {\bf Willmore surfaces as harmonic maps:} Consider the (incomplete) Willmore surfaces $Y_j\subset 
\RR^3_+\subset\mathbb{R}^3$, equipped with the isothermal coordinates $q_j\in [0,1], w_j\in [-1,1]$. For simplicity, 
denote $Y_j$ as $Y$ for the moment. Let $\ol{g}$, $\ol{\nabla}$ and $\ol{\Delta}$  be the induced Euclidean metric, 
connection and corresponding Laplacian. The Willmore surface in $\mathbb{R}^3$ determines a unique conformal harmonic map 
\[
\Phi:Y\to (dS_{1,3},h)\subset \mathbb{R}^{1,4}.
\]
Using coordinates $(t,x^1, x^2, x^3, x^4)$ so that $g_{\rm Mink}:=-dt^2+ \sum (dx^j)^2$, 
then $dS_{1,3} = \{-t^2+ \sum (x^j)^2=1\}$.  We recall first that
 \begin{equation}
\label{rel1}
\frac{1}{2}|\Abring|^2=(d\Phi)_\alpha^i(d\Phi)_\beta^j\ol{g}^{\alpha\beta}h_{ij} := |d\Phi|^2. 
\end{equation}

Since harmonic maps from $2$-dimensional domains are conformally invariant, we may as well
use the flat metric $g_\EE:=dq^2+dw^2$  on $Y$ rather than $e^{2\phi}(dq^2+dw^2)$. 
Observe that $|d\Phi|^2_{\EE} =e^{-2\phi}|d\Phi|^2$, so if $\phi$  is bounded above and below, then $|T|_\EE$ and $|T|$ are 
comparable; in particular, if $e^{|\phi|}\le \sqrt{2}$, then 
\begin{equation}
\label{rel1'}
\frac{1}{2}|\Abring|^2\le |d\Phi|^2_\EE\le 2|\Abring|^2;
\end{equation}
Now recall that
\begin{equation}
\label{rel2}
\Delta_\EE|d\Phi|_\EE^2=2|{\nabla} d\Phi|_\EE^2+(\mathrm{Riem}_{dS})_{ijkl}(d\Phi)_\alpha^i(d\Phi)_\beta^j
(d\Phi)_\gamma^k(d\Phi)^l_\delta (g_\EE)^{\alpha\gamma}(g_\EE)^{\beta\delta},
\end{equation}
which is the special case of the Bochner-type formula for any harmonic map \cite[Eqn. (8.7.13)]{Jost}). 
We recall the Riemann curvature tensor of deSitter space:
\[
(\mathrm{Riem}_{dS})_{ijkl}=(h_{ik}h_{jl}-h_{il}h_{jk}).
\]
Also, since $\Phi$ is conformal (and the metric induced by $h$ on $\Phi(Y)$ is Riemannian), then
$d\Phi(\del_q)$, $d\Phi(\del_w)$ are orthogonal and have the same length, hence
\[
\begin{split}
&h_{ik}h_{jl}(d\Phi)_\alpha^i(d\Phi)_\beta^j
(d\Phi)_\gamma^k(d\Phi)^l_\delta (g_\EE)^{\alpha\gamma}(g_\EE)^{\beta\delta}
\\&=2h_{il}h_{jk}(d\Phi)_\alpha^i(d\Phi)_\beta^j
(d\Phi)_\gamma^k(d\Phi)^l_\delta (g_\EE)^{\alpha\gamma}(g_\EE)^{\beta\delta}.
\end{split}
\]
From this, \eqref{rel2}, and the pointwise bounds on $\phi$ and $q/x$, we obtain that
\begin{equation}
\label{concln}
\Delta_\HH(e^{-2\phi}|\Abring|^2)  = (q_j)^2 \Delta_\EE |d\Phi|_\EE^2=2(q_j)^2|{\nabla} 
d\Phi|_\EE^2-(q_j)^2|d\Phi|_\EE^2 |d\Phi|_\EE^2, 
\end{equation}
where $\Delta_\HH$ is the Laplacian with respect to the hyperbolic metric $q_j^{-2}g_\EE$.  Hence, 
assuming pointwise bounds on $q_j^2|d\Phi|^2_\EE$ (which in fact hold for the surfaces $Y_j$, 
see Remark \ref{bd.comparison}), e.g.\ $q_j^2|d\Phi|^2_\EE\le 1$, we see finally that
\begin{equation}
\label{concln'}
\Delta_\HH(|d\Phi|^2_\EE) =(q_j)^2 \Delta_\EE |d\Phi|_\EE^2
 \ge 2(q_j)^2|\nabla d\Phi|^2_\EE-2|d\Phi|_\EE^2.
\end{equation}

\medskip

\noindent {\bf Modified weight function and the isothermal parametrization:}  We now
recall the weight function $f$.  After isolating the part of each $Y_j$ in the original Willmore
surface $\hat{Y}_j$ which admits good isothermal coordinates, some poles lie $Y_j$ and others lie 
in $\hat{Y}_j\setminus Y_j$. We shall modify the weight function $f$ to omit those poles which do 
not lie in $Y_j$, but to justify how we do this we require some preliminary estimates.
\begin{lemma}
\label{big-small-discs}
Consider $Y\subset\mathbb{H}^3$, and assume that the connected component of $Y\bigcap B_+(0,3)$ containing
the origin is a graph $z=u(x,y)$ over the half-disc $D_+(3)$, with $|\nabla u|\le 1$. Then for points $B\in Y \setminus
{\mathrm Graph}(u)$ and $A_x :=(x,0,u(x,0)) \in {\mathrm Graph}(u)$, $A_x$, we have 
\[
d_{g} (B,A_x)\ge |\log x|. 
\]
\end{lemma}
\begin{proof} The proof is elementary: if $B = (x_0,y_0,z_0)$, then set $\tilde{B} = (x_0,y_0,0)$ and $\tilde{A}_x = (x,0,0)$. 
It suffices to check that $d_{\mathbb{H}^3}(\tilde{B},\tilde{A}_x)\ge |\log x|$. The geodesic $\gamma(\tilde{B},\tilde{A}_x)$ 
joining $\tilde{B},\tilde{A}_x$ is a circular arc. If $x_0\ge 1$, the claim is obvious since $d_{\mathbb{H}^3}(\tilde{B},\tilde{A}_x)\ge 
d_{\mathbb{H}^3}((x_0,0,0),\tilde{A}_x)\ge |\log x|$. If $x_0\le 1$, however, then since $(x_0)^2+(y_0)^2\ge 4$, this circular arc 
must intersect the line $\{x=1, z=0\}$, we can apply the previous argument. 
\end{proof}
As an immediate consequence, 
if $B\notin D_+(0,2)$, then $d_{\mathbb{H}^2}(B,A_x) \ge | \log x|$. 
   
Now consider the {\it hyperbolic} metric $\g_\HH:=q_j^{-2}( (dq_j)^2+(dw_j)^2)$ on $Y_j$. This is conformal to the 
metric $g_{j}$ induced by the embedding $Y_j \subset \HH^3$; indeed, 
\[
\frac{x^2e^{-2\phi_j}}{(q_j)^2}g_{j}=g_\HH.
\]
Quantities computed with respect to this metric will be labelled with a $\HH$. In particular,
with ${\cal B}_j = \{0\le q_j\le 1, -1\le w_j\le 1\} \subset Y_j$, then for any $P\in {\cal B}_j$, we write $B^R(P)$ 
and $B^R_\HH(P)$ for the balls around $P$ of radius $R$ with respect to $d_{g_j}$ and $d_\HH$. 

The bounds on $\sup |\phi_j|$ and $\sup |\nabla q_j|$ give upper and lower bounds on $q_j/x$, which imply that 
\begin{equation}
\label{distortion}
|d_\HH(P,Q)-d_{g_j}(P,Q)|\le 1 \qquad \forall\ P,Q\in Y_j.
 \end{equation} 

Lemma \ref{big-small-discs} implies that after rebalancing, then for each point $P\in \{0\le q_j\le 1, w_j=0\}$ 
and each pole $O_k\in \hat{Y}_j\setminus Y_j$, we have
\begin{equation}
\label{cheap.bound}
d_{g_j}(O_k,P)\ge |\log x(P)|,
\end{equation}
and hence also
\begin{equation}
\label{cheap.bound'}
d_{g_j}(O_k,P)+1\ge |\log q_j(P)|,
\end{equation}  
using the upper and lower bounds on $\frac{q_j}{x}$ in ${\cal B}_j$.

These considerations make it natural to modify the weight function $f$ slightly. Thus, define the new function 
$\tilde{f}_j$ on ${\cal B}_j$ by setting $\tilde{f}_j(Q)=f_j(Q)$ if $Q\sim O_k$, provided $O_k\in {\cal B}_j$, and 
$\tilde{f}_j(Q):=|\log (q_j(Q))|+5$ otherwise. Denote the weighted energy associated to $\tilde{f}$ by $\tilde{\cal E}_p$,
and observe that by \eqref{cheap.bound} and \eqref{cheap.bound'}, 
\[
\tilde{\cal E}_{p}[Y_j] =\int|\Abring_j|^2 \tilde{f}_j^2d\mu_j\le 2{\cal E}_p[Y_j].
\]
\begin{proposition}
\label{Abound}
On the segment $\{0\le q_j\le 1, w_j=0\}$, we have
\[
|\Abring_j|_\EE(P)\le 4\frac{\sqrt{\tilde{\cal E}_p^{B^1_{\HH}(P)}[Y_j]}}{\tilde{f}_j^{p}(P)q_j(P)}.
\]
\end{proposition}
\noindent This will be proved in the next subsection. 

We now check how this proposition implies \eqref{eqndecay}. Assume (passing to a subsequence) that there 
are $K$ poles $O_i$ in ${\cal B}_j$. There is an obvious bound 
\begin{equation}
\label{trivial.bound}
\begin{split}
&1/\tilde{f}_j^{p}(P) \le \frac{N-K}{(-\log q_j(P)+5)^{p}}+\frac{1}{(\min_{i \leq K} \mathrm{d}_{g_j}(P,O_i)+5)^{p}}
\\&\le \frac{N-K}{(-\log q_j(P)+5)^{p}}+\sum_{i=1}^K\frac{1}{({\rm d}_{g_j}(P,O_i)+5)^{p}}. 
\end{split}
\end{equation}

Observe that 
\[
 \int_0^1 \frac{1}{q(|\log q| +5)^{p}} dq=\frac{5^{1-p}}{p-1}.
\]
Hence it suffices to obtain uniform bounds for the individual integrals 
\[
\int_0^1\frac{1}{q_j[{\rm d}_{g_j}(\cdot,O_i)+5]^{p}} \, dq_j, \ i \leq  K. 
\]

Consider the poles $O_i \in{\cal B}_j$ and set $q_j(O_i):=\delta_{j,i}\in (0,1]$. Equation \eqref{distortion} implies
that $d_{g_j}(P,O_i)+5\ge |\log q_j(P)-\log\delta_{j,i}|+4$, so we finish the proof by noting that
\[
\int_0^1\frac{1}{q_j (|\log q_j-\log\delta_{j,i} |+4)^{p}} \, dq_j \leq \frac{2}{p-1}.
\]

\subsection{Proof of Proposition \ref{Abound}}
We begin by recalling a mean value inequality 
\begin{lemma}[\cite{Morrey}]
\label{Mean.value}
There exists a constant $C>0$ such that if $F$ is any positive $\calC^2$ function on the ball $B^1_\HH(P)\subset \HH$ 
satisfying $\Delta_\HH F\ge -2F$, then 
\[
F(P)\le C\int_{B^1_\HH(P)}F\, d\mu_\HH.
\]
\end{lemma}
To use this, we modify $\tilde{f}_j$ slightly further: suppose that $P\in Y_j$ and $P\sim O_k$; if 
$O_k\in {\cal B}_j$ we define $\ol{f}_j(Q):= d_\HH(Q,O_k)+5$ for $Q\in B^1_\HH(P) $, while if 
$O_k\notin{\cal B}_j$ (so $\tilde{f}_j(P)=-\log (q_j(P))+5$) then $\ol{f}_j(Q) = - \log(q_j(Q))+5$ for 
$Q \in B^1_\HH(P)$. In other words, if $Q\in B^1_\HH(P)$, then $\ol{f}_j(Q)$ either equals the $\HH$-distance to 
the pole closest to $P$, or else, if the nearest pole does not lie in ${\cal B}_j$, it equals $-\log q_j+5$. 

Observe that $\tilde{f}\le \ol{f}+1\le 2\ol{f}$ on $B^1_\HH(P)$.  To compare these functions in the 
other direction, suppose $Q\in B^1_\HH(P)$, with $P\sim O_k$ and $Q\sim O_r$. Using \eqref{distortion} and the 
triangle inequality, we find 
\begin{multline*}
\ol{f}(Q)-5=d_\HH(Q,O_k)\le d_\HH(Q,P)+d_\HH(P,O_k)\le 2+d_{g_j}(P,O_k)\\ 
\le 2+d_{g_j}(P,O_r)\le 2+d_{g_j}(P,Q)+d_{g_j}(Q,O_r)\le 3+d_{g_j}(Q,O_r)\le \tilde{f}(Q), 
\end{multline*}
and hence 
\begin{equation}
\label{fbd}
\ol{f}_j(Q)\le 2\tilde{f}_j(Q).
\end{equation}
One consequence is that 
\begin{equation}
\label{subst.energy}
\int_{B^1_\HH(P)}|A_j|^2\ol{f}_j^{2p}\, d\mu_{\HH}\le 4\int_{B^1_\HH(P)}|A_j|^2\tilde{f}_j^{2p}\, d\mu_{\HH}.
\end{equation}
In any case, we have proved that $1/2 \leq |\ol{f}_j/\tilde{f}_j| \leq 2$ in $B^1_\HH(P)$, and so it suffices to
prove Proposition \ref{Abound}  with $\ol{f}_j$ replacing $\tilde{f}_j$. 

We now prove the Proposition.  We first claim that $(\Delta_\HH \ol{f}_j)\ol{f}_j-3|\nabla \ol{f}_j|^2_\HH\ge 0$. 
Indeed, in the region where $\ol{f}_j=d_\HH(O_k,\cdot)+5$, then $|\nabla \ol{f}_j|_\HH=1$, and the 
differential inequality follows from the standard formula $\Delta_\HH \ol{f}_j=\coth(\ol{f}_j-5)$. On the other hand, 
when $\ol{f}_j=-\log q_j+5$ then it follows by calculating that
\[
(\Delta_\HH \ol{f}_j)\ol{f}_j-3|\nabla \ol{f}_j|^2_\HH=(-\log q_j+5-3)\ge 0,
\]
since $q_j\in (0,1]$. In any case, by \eqref{concln'} and Cauchy-Schwarz, since $p\in (1,2]$, 
\begin{equation}
\label{calcn}
\begin{split}
\Delta_\HH & (|d\Phi_j|_\EE^2 \ol{f}_j^{2p})  =  q_j^2\Delta_\EE (|d\Phi_j|_\EE^2\ol{f}_j^{\, 2p})  \\[0.5ex] 
& = q_j^2\ol{f}_j^{\, 2p-2}\bigg(\Delta_\EE (| d\Phi_j|_\EE^2 )\ol{f}_j^2  + 2p |d\Phi_j|^2_\EE (\Delta_\EE \ol{f}_j)\ol{f}_j \\ & 
+2p(2p-1)|d\Phi_j|_\EE^2|\nabla \ol{f}_j|_\EE^2+8p\nabla_s (d\Phi_j)_r^\alpha (d\Phi_j)_t^\beta h_{\alpha\beta}(g_\EE)^{rt}
\nabla^s \ol{f}_j \ol{f}_j\bigg) \\ 
& \ge p\ol{f}_j^{\, 2p-2}\big(-6|d\Phi_j|_\EE^2|\nabla \ol{f}_j|^2_\HH+2|d\Phi_j|_\EE^2(\Delta_\HH \ol{f}_j)
\ol{f}_j\big)-2|d\Phi_j|_\EE^2\ol{f}_j^{\, 2p} \\[0.5ex] & \ge -2|d\Phi_j|_\EE^2\ol{f}_j^{\, 2p}. 
\end{split}
\end{equation}

Using Lemma \ref{Mean.value} and (\ref{rel1'}), there exists a universal constant $C>0$ such that 
\begin{equation}
\label{Monot.conseq}
\begin{split}
\frac14 |\Abring_j(P)|_\EE^2\ol{f}_j^{\, 2p}(P) & \le |d\Phi_j(P)|^2_\EE \ol{f}_j^{\, 2p}(P) \\ & 
\le C \int_{B^1_{\HH}(P)} |d\Phi_j|^2_\EE \ol{f}_j^{\, 2p} \, d\mu_\HH \le 4C \int_{B^1_{\HH} (P)} |\Abring_j|_\EE^2 \ol{f}_j^{\, 2p}\, d\mu_\HH.
\end{split}
\end{equation}
Now, since $e^{-2}\le \frac{q_j(Q)}{q_j(P)}\le e^2$ for $Q\in B^1_\HH(P)$ and $|\phi|$ is uniformly bounded, 
then
\[
|\Abring_j|_\EE^2=\frac{e^{-2\phi_j}}{q_j^2}|\Aring_j|_{g_j}, \quad \mbox{and}\quad
d\mu_\HH=e^{-2\phi_j}\frac{x^2}{q_j^2}\, d\mu_j
\] 
implies that
\begin{equation}
\label{final}
|\Abring_j(P)|_\EE\le C\frac{\sqrt{\int_{B^1_\HH(P)} |\Abring_j|_\HH^2\ol{f}_j^{2p}d\mu_\HH}}{q_j(P)\ol{f}_j^p}\le 
C\frac{\sqrt{\int_{B^{2}(P)} |\Abring_j|_\HH^2\ol{f}_j^{2p}d\mu_\HH}}{q_j(P)\ol{f}_j^p}.
\end{equation} 
This finishes the proof of the proposition \ref{decay}.  
     
\subsection{Proof of Proposition \ref{bound.phi.conseq}: the second line integral}
\label{flux.integral.bound}
The analysis of the second line integral in \eqref{big} differs from that of the first one. In particular, rather than deriving 
pointwise control for the integrand (which appears hopeless), we express the entire integral as a flux of a suitable vector 
field across the line $\{0\le q_j\le 1, w_j=0\}$ using Stokes' theorem. The bound is then obtained by 
controlling the integral of the divergence over boxes adjacent to this line.

Since $e^{2\phi_j}=\olg_j(\partial_1,\partial_1)$, we obtain that for $j$ large enough, $|\phi_j|\le 1/10$, and hence
\[
e^{|\phi_j|} \le 2, \quad \mbox{and} \qquad 1/2\le |\partial_1|_{\olg_j} \le 2.
\]
From these bounds we also obtain
\begin{equation} \label{simple.bound} 
|\nabla_1 u_j|\le 4\zeta \Longrightarrow  -\frac12 \leq \del_1u_j \leq 2.
\end{equation}
Recall also the basic equation, which follows from the Codazzi formul\ae, 
\begin{equation} \label{key.eqn}
-\Delta_{\olg} \, \phi=\frac{4\overline{H}^2-|\Abar|^2}{4}=\frac{4\overline{H}^2-|\Abring|^2}{8},
\end{equation}
as well as the identity 
\begin{equation}
\label{Lapl.break}
\Delta_{\olg}\, e^{-\phi}=-\Delta_{\olg}\phi e^{-\phi}+|\nabla\phi|_{\olg}^2\, e^{-\phi}.
\end{equation} 

The key for proving \eqref{second.bound} is to express the integral $I$ on the left in that formula 
as one of the boundary flux terms of an integration of the divergences of two vector fields over the two rectangles 
$\calD_1:=\{0\le q_j\le 1, 0\le w_j\le 1\}$ and $\calD_2:=\{0\le q_j\le 1, -1\le w_j\le 0\}$.
To do this, introduce a cutoff function $\chi(w)$ such that $\chi\in\calC^2$ with $0\le \chi\le 1$, 
$\chi(0)=1,\chi(-1)=\chi(1)=0$, and such that $|\chi'|\le 4$.

By Stokes' formula, and with summation over $s$ implied, 
\begin{equation}
\begin{split}
\label{flux.expression}
&I= \int_{{\cal D}_1} \Delta_{\overline{g}_j}e^{-\phi_j}(\partial_1u_j+1)\chi \, d\overline{\mu}_j + \int_{{\cal D}_1}\partial_se^{-\phi_j}
\partial_s (\partial_1u_j+1)\chi\,  dq_jdw_j \\
& +\int_{\calD_1}\del_s(e^{-\phi_j})(\partial_1u_j+1)\partial^s\chi\, d\overline{\mu}_j - \left. \int\del_1e^{-\phi_j}(\del_1u_j+1)\chi \,dw_j\right|_{q_j=0}^{q_j=1}\\
&+\int_{\calD_2}\Delta_{\overline{g}_j} e^{-\phi_j}  \chi \,d\overline{\mu}_j+ \int_{\calD_2} \del_s e^{-\phi_j}e^{\phi_j}\partial_s\chi \, dq_jdw_j 
-\left. \int\partial_1e^{-\phi_j}e^\phi\chi\, dw_j\right|_{q_j=0}^{q_j=1} 
 \end{split}
 \end{equation}
Some of these integrals are expressed with respect to the $d\bar{\mu}_j$ others
 with respect to the volume form $dq_jdw_j$, 
but the difference is not large since $d\bar{\mu}_j=e^{2\phi_j}dq_jdw_j$. 

Since $Y_j\to Y_*$ smoothly away from $\{x=0\}$, then for {\it any} $\eta \in (0,1]$, $\partial_{q_j}|_{x=\eta}
\rightarrow \partial_{q_*}|_{x=\eta}$ and $\phi_j|_{x=\eta}\rightarrow \phi_*|_{x=\eta}$ smoothly. 
The bounds on $|\del_1|_{\olg}$ and the fact that $\phi_* = \mbox{const}$ shows that 
\begin{equation}\label{easy.bound}
\left. \int_{-1}^{1}|\del_1\phi_j| \, dw_j \right|_{q_j=1} =o(1).
\end{equation} 
On the other hand, by \eqref{the.bounds}, 
\begin{equation}
\label{L1.in.a.box}
\int_{0}^1\int_{-1}^1 |\partial_{12}\phi_j|dw_jdq_j\le {\cal E}_j+o(1).
\end{equation}
Combining these last two equations, we see that if $\eta \in (0,1]$,  then for $-1 \leq a\le w\le b \leq 1$, 
\begin{equation}
\label{2.and.2.together}
\left|\int_a^b\del_1\phi_j\, dw_j\right|_{q_j=\eta} \le \left|\int_a^b\del_1\phi_j\, dw_j \right|_{q_j=1}+\int_\eta^1
\int_a^b|\del_{12}\phi_j|\, dw_jdq_j. 
\end{equation}
Since this is true for all subintervals $[a,b]$, then for each $\eta$ we can divide the integral on the left into subintervals 
$[a,b]$ where $\partial_1\phi_j|_{q_j=\eta}$ has constant sign, and then add these subintervals, to obtain that 
\begin{equation}\label{hardy.preconseq}
\left. \int_{-1}^{1}|\del_1\phi_j| \,dw_j \right|_{q_j = \eta} \le {\cal E}_j+o(1),
\end{equation} 
where the error term is independent of $\eta$. Letting $\eta\rightarrow 0$ gives 
\begin{equation}\label{hardy.conseq} 
\left. \int_{-1}^{1}|\del_1\phi_j| \,dw_j \right|_{q_j = 0} \le {\cal E}_j+o(1).
 \end{equation} 

Now consider the interior integral terms in \eqref{flux.expression}. By \eqref{Lapl.break}, the  
first interior integral can be written as 
\[
\int_{\calD_1} -\Delta_{\olg_j}\phi_j(e^{-\phi_j})(\partial_1 u_j+1)\chi \, d\bar{\mu} + 
\int_{\calD_1} |\nabla\phi_j|_{{\olg_j}}^2 \, e^{-\phi_j}(\del_1 u_j+1) \chi \, d\bar{\mu}. 
\]
The first term on the right here is controlled using \eqref{key.eqn}: 
\begin{equation}
\label{first.bulk}
\begin{split}
-\int_{\calD_1}\Delta_{\olg_j}\phi_j e^{-\phi_j}(& \del_1 u_j+1 ) \chi \,  d\bar{\mu} =  
\int_{\calD_1} \frac{|\Abring_j|^2-4\overline{H}_j^2}{8}(\del_1 u_j +1 )\chi e^{-\phi_j}\, d\bar{\mu} \\
& \le 4\int_{\calD_1} \left(2|\Abring_j|^2-\frac{\overline{H}^2_j}{64} \right) \chi e^{-\phi_j}\,  d\bar{\mu}.
\end{split}
\end{equation}
    
Next, define
\[
{\cal T}:=\int_{\calD_1} |\nabla\phi_j|_{\olg_j}^2 e^{-\phi_j} (\del_1 u_j +1) \chi d\bar{\mu} 
+\int_{\calD_1} \del^s e^{-\phi_j}\del_{s1} u_j \chi  \, d\bar{\mu}.
\]
Replace $\partial_{1s} u_j$ by $\nabla_{1s}u_j$ using $\nabla_{ab}u_j=\partial_{ab}u_j-\Gamma_{ab}^t\partial_tu_j$, where
\begin{equation}
\label{christoffels}
\Gamma_{21}^1=\partial_1\phi_j, \Gamma_{22}^2=\partial_2\phi_j, \Gamma_{21}^2=-\partial_2\phi_j,
\Gamma_{22}^1=\partial_1\phi_j, \Gamma_{12}^1=\partial_2\phi_j, \Gamma_{12}^2=\partial_1\phi_j,  
\end{equation}
to get that $\cal T$ equals
\begin{equation}
\label{simplify.1}
\begin{split}
&\int_{{\calD_1}} |\nabla\phi_j|_{\olg_j}^2 e^{-\phi_j} (\nabla_1 u_j+1) \chi \, d\bar{\mu} 
+\int_{\calD_1} \del^s e^{-\phi_j}\nabla_{s1} u_j \chi  \, d\bar{\mu}  \\ 
& +  \int_{\calD_1} \del^s e^{-\phi_j}\Gamma_{s1}^t\partial_t u_j\chi  \, d\bar{\mu} 
= \int_{\calD_1} \left( \del^se^{-\phi_j} \nabla_{s1} u_j \chi +|\nabla \phi_j|^2_{\olg_j} e^{-\phi_j}\chi \right)\, d\bar\mu. 
\end{split}
\end{equation} 
Applying Cauchy-Schwarz, \eqref{the.bounds} and $|\nabla_{ab} u_j|_{\olg_j}\le 2|(\Abar_j)_{ab}|_{\olg_j}$ we derive:
\begin{equation}
\label{concln.simplify}
{\cal T}\le 100\int_{\calD_1\cup \calD_2}  |\nabla\phi_j|^2\chi \, d\bar{\mu}
+ \frac{1}{100}\int_{\calD}|\Abar_j|^2 \chi\, d\bar{\mu} \le C200 \, \calE_j+\frac{1}{50}
\int_{\cal D}|\overline{H}_j|^2d\bar{\mu}+o(1). 
\end{equation}
As for the third bulk term, using \eqref{the.bounds} again, we derive
\begin{equation} \label{third.bulk}
\begin{split}
&{\cal Z}:=\int_{\calD_1} e^{-2\phi_j}\del_2e^{-\phi_j}(\del_1 u_j+1)\del_2\chi\, d\bar{\mu} 
=\int_{\calD_1} \del_2e^{-\phi_j}(\del_1 u_j+1)\del_2\chi\, dqdw \\
&\le 4\sqrt{2\int_{\calD_1} |\nabla \phi_j|_{\olg_j}^2d\bar{\mu}}\cdot 
\sqrt{\int_{\calD_1}\, dqdw} \le  10C\sqrt{\calE_j+o(1)}. 
\end{split}
\end{equation}
We control the last two bulk terms by 
\begin{equation} \label{penultimate.bulk}\begin{split}
&\int_{\calD_2} (- \Delta_{\olg_j} \phi e^{-\phi_j}+|\nabla\phi_j|_{\olg_j}^2e^{-\phi_j}) \chi \, d\bar{\mu} \le 
\int_{\calD_2} \frac{|\Abring_j|^2-\overline{H}_j^2}{4}|\chi|\, d\bar{\mu}+2C\calE_j+o(1). 
\end{split} 
\end{equation}
Finally, the Cauchy-Schwarz inequality together with \eqref{the.bounds} one last time gives
\begin{equation} \label{ultimate.bulk}
\int_{\calD_2} e^{-2\phi}_j\del_2\phi_j \del_2\chi \, d\overline{\mu} \le 2\sqrt{\int_{\calD} 
|\nabla e^{-\phi_j}|^2\, d\bar{\mu}}\, \sqrt{\int_{\calD_2} \, dqdw}
\le 2\sqrt{\int_{\calD_2} |\Abring_j|^2\,d\bar{\mu}}.
\end{equation}

Taken together, these estimates complete the proof. The only thing to observe is that the terms 
$\int_{{\cal D}}\overline{H}_j^2 \,d\bar{\mu}$ appears with a {\it negative} coefficient in the end, and so can be discarded,
since our proposition only claims an upper bound on $I$. \hfill $\Box$

\section{Regularity gain for the limit surface in the small energy regions}
\label{C1bdry}
We now turn to a closer look at the relationship between finiteness of the weighted energy  and the regularity of the boundary
curve at infinity, and prove Theorem \ref{removability}. In fact we prove the $\calC^1$ regularity for all Willmore 
surfaces with finite weighted total curvature near points where the boundary curve is locally graphical and Lipschitz. 
\begin{definition}
\label{loc.Lipsch} 
Consider a rectifiable, closed, embedded loop $\gamma\subset \RR^2$, with arclength parametrization 
$t\rightarrow (y(t),z(t))=\gamma(t)$. We say that $\gamma$ is locally Lipschitz at $P=\gamma(t_0)$
if there exists a $\delta(t_0)>0$ and a constant $M(t_0)<\infty$ such that (after a rotation), the portion of $\gamma$
parametrized by $(t_0-\delta,t_0+\delta)$ coincides with the graph $z=f(y)$ over an interval of length $\eta(t_0)$ 
centered at $\gamma(t_0)$. Thus $\gamma((t_0-\delta,t_0+\delta))={\rm Graph}(f)$ and $|f(y_1)-f(y_2)|\le M(t_0)|y_1-y_2|$. 
\end{definition}

Our main result in this section is the 
\begin{theorem}
\label{gen.regularity} Let $Y\subset \HH^3$ be a complete Willmore surface with $\gamma = \partial_\infty Y$
a possibly disconnected embedded rectifiable curve. Suppose that there exists a set of poles 
$\calO = \{O_1,\dots, O_K\}\subset Y$ such that  $\calE_p(Y)<\infty$ (the weight function $f$ relative to $\calO$ is implicit)
and that $\gamma$ is locally graphical and Lipschitz except at a finite number of points $\{P_1,\dots, P_\Lambda\}$.  
Assume finally that if $\gamma(t_0) \neq P_j$ for any $j$, $M(t_0)=\zeta$ and 
$\calE_p^{B(\gamma(t_0), \delta(t_0))}(Y)\le\e'(\zeta)$. Then $\gamma\setminus \{P_1,\dots, P_\Lambda\}$ is a $\calC^1$ curve. 
\end{theorem}
In the setting in Theorem \ref{removability}, the assumption that the boundary curve is locally Lipschitz away from 
$\{P_1\dots, P_\Lambda\}$ holds for the curve $\del_\infty Y_*$ which is the limit of the $\del_\infty Y_j$. 
Indeed, Corollary \ref{ereg.surface} ensures the graphicality and Lipschitz bound away from the bad points 
$P_1,\dots P_\Lambda$. We distinguish two further cases. Either there exists a sequence of poles $O_k^{(j)}$ 
converging to an interior point $O_*\in Y_*$, or else any sequence of poles $O_k^{(j)}$ diverges to infinity 
in the limit. In the first case, without precluding that some poles disappear to infinity, suppose that the limits 
of the poles occur at $O_{*,1},\dots, O_{*,K} \in Y_*$.  We can also assume that the poles $O_k^{(j)} \in Y_j$ 
converge to $O_{*,k}\in Y_*$. Also, using the weight function $f_*$ on $Y_*$ corresponding to the poles 
$\{O_{*,1},\dots, O_{*,K}\}$,  the weighted Willmore energy is finite. To see this, note that for all $\epsilon>0$ 
\[
\calE_p(Y_j\cap \{x\ge \epsilon\})\to \calE_p(Y_*\cap \{x\ge \epsilon\}).
\]
This follows readily since all poles other than $O_1,\dots, O_K$ disappear towards infinity, thus $f_j \to f_*$
over the portion of the surfaces contained in $\{x\ge \epsilon\}$. So consider the second case, where 
$O_k^{(j)}\to \del_\infty \HH^3$ for all $k$. We claim that $Y_*$ must then be a finite union of half-spheres;
this implies Theorem \ref{removability} immediately. To prove this assertion, just note that
if all poles disappear towards infinity then $|\Abring|=0$ on all of $Y_*$: If this were false, then there 
would exist an interior point $P\in Y_*$ 
and a ball $B^1(P)\subset Y_*$ such that $\int_{B^1(P)}|\Aring|^2d\mu=\epsilon>0$. But then, 
since all the $O_k^{(j)}$ diverge to infinity, $f|_{B^1(P)}\to\infty$ on $B^1(P)$, which implies that 
$\calE_p[Y_j]\to\infty$ as well. This is a contradiction.

Therefore we have reduced to proving Theorem \ref{gen.regularity}. This, in turn, is a consequence of the 
following result. 
\begin{proposition}
\label{Lip.C1} 
Let $\gamma_k(t)$, $0 < t < M_k$, be an arclength parametrization of the $k^{\mathrm{th}}$ connected component of 
$\gamma$. Suppose that $\gamma(t_*)\notin \{P_1,\dots, P_\Lambda\}$.
Choose any Cauchy sequence $t_j\in (0,M_k)$ where $\gamma_k$ is differentiable at $t_j$,
 with $t_j\to t_*  \in (0,M_k)$. Then $\dot{\gamma}(t_j)$ 
is a Cauchy sequence. 
\end{proposition}
\begin{proof} Since $\gamma(t_*)$ is not equal to one of the bad points $P_j$, there exists a line $\ell$ through $\gamma_k(t_*)$ 
and a number $\delta$ such that $\gamma|_{(t_*-\delta,t_*+\delta)}$ is a graph over the interval of length $\delta$ 
in $\ell$ centered at $\gamma(t_*)$ with graph function $z=f(y)$ having Lipschitz constant $\zeta$. 
Lemma \ref{Lipschitz.lift} guarantees graphicality of $Y'_{B(\gamma(t_0), h \, \delta)}$ over the region $\sqrt{x^2+y^2}\le h\delta$
in the vertical half-plane $\ell \times \RR^+$ (where we take $\ell$ as the $y$-axis), with graph function $z=u(x,y)$,
where $|\nabla u|\le 2\zeta$.

Since $t_j$ is Cauchy, it lies in $(t_*-h\, \delta,t_*+h\, \delta)$ for $j$ large, so if we write $\gamma(t_j) = (y_j, u(0, y_j))$, 
then $y_j\rightarrow 0$. 

Now, argue by contradiction and assume that $\dot{\gamma}(t_j)$ is not Cauchy. Then there exists $\theta>0$ 
and a subsequence $j_k$ such that $|f'(y_{j_{k-1}})-f'(y_{j_k})|\ge \theta$.  Reset notation so that the index is 
simply $j$ again. Translate and rotate so that $(y_{j-1}, f(y_{j-1})) = (0,0)$ and $(y_j,f(y_j))$ lies on the $y$-axis,
then dilate by the factor $\lambda_j := |y_j-y_{j-1}|^{-1}$. Denote the resulting Willmore surface by $\tilde{Y}_j$ and write 
$\del \tilde{Y}_j = \tilde{\gamma}_j$. 

This surface is still graphical with Lipschitz norm no larger than $\zeta$, and furthermore, 
$\calE^{B(0, \lambda_j h \delta)}(\tilde{Y}_j) \le\e'(\zeta)$. By Lemma \ref{no.spheres}, $Y_j$ must 
converge to a vertical half-plane $Y_*$, and since $\del_\infty Y_*$ passes through the origin and $(0,1,0)$, 
necessarily $Y_* = \{z=0\}$. Thus $Y_*\cap \{x=1\}$ must converge to the line $\{z =0, x =1\}$
for some $\alpha$ with $|\alpha|\le 2\zeta$. But now, since $|f_j'(0)-f_j'(1)|\ge \theta$, it follows that 
for at least one of the two values $y=0, y=1$ there is a jump in the derivative of size at least $\theta/2$ between 
the heights $x=0$ and $1$. We can assume that this jump occurs at $y=0$. 

However, this contradicts Proposition \ref{ereg.prp2}.  The graphicality and Lipschitz bound in that Proposition still
hold by virtue of the assumption and Lemma \ref{Lipschitz.lift}. The fact that the weighted energy goes to zero follows 
from the dilation invariance of $\calE_p$, and the fact that after dilation, the graphs satisfy 
$\calE_p(\mathrm{Graph}(u_j))\le\calE_p(Y\cap B(P,2h \lambda_j^{-1})]$, and $\lambda_j^{-1} = |y_j-y_{j-1}|\to 0$. 
This proves the Proposition and Theorem~\ref{gen.regularity} as well.
\end{proof}

 \section{Bubbling in the small energy regions.}
We now turn to a closer examination of how bubbling occurs, aiming toward the proof of Theorem \ref{jump.implies.bubble}. 

\smallskip
The argument leading to the fact that bubbling occurs is indirect. We first construct a sequence of M\"obius 
transformations $\varphi_j$ to obtain  uniform isothermal 
parametrizations for the surfaces $\varphi_j(Y_j)$. If the surfaces $\varphi_j(Y_j)$ 
converge to a non-trivial surface, we are done. Otherwise, we must prove that one can take a further
sequence of dilations to obtain a nontrivial limit. 

The idea is to use the jump in the first derivative coupled with the bounds \eqref{nabla.ext.curv} to argue
that one of the two line integrals on the right side of that equation must be bounded below.  In particular, 
with $4\epsilon_0:=\ lim_{j\to\infty} \partial_yu_j(0,0)-\partial_yu_*(0,0)$, then by \eqref{nabla.ext.curv}, either
\begin{equation}
\label{eqn1}
\int_0^1|\Abring_{12}|e^{-\phi_j}\, dq_j\ge \epsilon_0\text{ } \text{or}\text{ } 
 \end{equation}
 \begin{equation}
\label{eqn2}
\left|\int_0^1\partial_2e^{-\phi_j}\del_qu_j\, dq_j \right|\ge\epsilon_0.
 \end{equation}
These cases are treated separately in the next two subsections. We also show, in \S \ref{finitedistance},
that each bubble remains at finite distance from one of the poles. 

\subsection{The integral $\int_0^1|\Abring_{12}|e^{-\phi_j}dq_j$ bounded below implies bubbling.}
Assuming (\ref{eqn1}), from Proposition \ref{decay}, we derive that
\[
\sup_j \sup_{P\in \{0\le q_j\le 1, w_j=0\}}{\cal E}_p(B^1_\HH(P))\ge \frac{(p-1)}{C(K)M'}\epsilon_0.
\]
Consider the set of points $P\in \{0\le q_j\le 1, w_j=0\}$ where ${\cal E}_p (B_\HH^1(P))\ge\frac{(p-1)\epsilon_0}{10C(K)M'}$. 
We know that such points exist when $j$ is large.  We then ask whether there exists a constant $M$ and 
points $P_j\in\{0\le q_j\le 1, w_j=0\}$ with ${\cal E}_p (B_\HH^1(P_j))\ge \frac{(p-1)\epsilon_0}{10C(K)M'}$
such that $f_j(P_j)\le M$. (The requirement $f_j(P_j)\le M$ is equivalent to the existence 
of $M<\infty$ such that ${\rm dist}(P_j, O_k^{(j)})\le M$, where $P_j\sim O_k^{(j)}$).

If such a sequence $P_j$ exists, then consider the isometry $\varphi_j$ of $\HH^3$ which maps $P_j$ to 
$P_*=(1,0,0)$. The surfaces  $\varphi_j(Y_j)$ converge (in a  large closed ball around $P_*$) to a 
Willmore surface $Y_*$ with ${\cal E}(B^1_\HH(P_*),Y_*)\ge \epsilon_0/10M^2$, and this would be the
desired `bubble'. 

It suffices then to show that \eqref{eqn1} must fail if no such sequence $P_j$ exists. Indeed, 
observe that if $1<p'<p$, then Proposition \ref{Abound} gives the new bound
\[
|\Abring_j|_\EE(P)\le \frac{C_{p'}\sqrt{{\cal E}_{p'}^{B^1_\HH(P)}(Y_j)}}{q_j(P)(\tilde{f}^{p'}_j(P))}, \quad
\mbox{whence}\quad
\int_0^1 |\Abring_j|_\EE \, dq\le C'_{p'} \sup_{P\in l_j}\sqrt{{\cal E}_{p'}^{B^1_\HH(P)}[Y_j]}.
\]
However, observe that  $\limsup_{j\to\infty} \sup_{P\in l_j}{\cal E}_{p'}(B^1_\HH(P))=0$; this holds
because by definition  $\calE_{p'}^{B^1_\HH(P)}(Y_j) \le 2f_j^{p'-p}(P){\cal E}_{p}^{B^1_\HH(P)}(Y_j)$, and 
$\calE_{p}^{B^1_\HH(P_j)}(Y_j)\le M''$, while $f_j(P_j)\to \infty$. Taken together, this all shows that
\[
\limsup_{j\to\infty}\int_0^1 |\Abring_j|_\EE \, dq_j=0, 
\]
contradicting \eqref{eqn1}, as claimed. 

\subsection{A lower bound on the flux \eqref{eqn2} implies bubbling}
Our goal is to show that such  a lower bound  (\ref{eqn2})  implies the existence of further blow-ups $\varphi_j:\mathbb{H}^3\to\mathbb{H}^3$ 
such that $\varphi_j(Y_j)\to Y_*$ with $\calE[Y_*]>0$. Unlike in the
final subsection of \S 5, it is not enough to bound the line integral 
 $|\int_0^1\partial_2e^{-\phi_j}\partial_qu_jdq_j|$ by the energy in a box. Fortunately,
we can bound it in terms of the energy in a sector $|\frac{w}{q}|\le 1$ emanating from the distinguished
boundary point in the isothermal coordinates $(q,w)$.  This bound can then be used to show the
existence of a sequence of points where either $|x \Abar_j|$ or $|\nabla \phi_j|_g$ are bounded away from zero. 
Either alternative provides the points around which we can recenter the rescalings. In the first case, we obtain 
a limit surface with non-zero curvature at one interior point, which must therefore be nontrivial. In the second we 
obtain a complete Willmore surface for which the canonical isothermal coordinates have non-constant conformal 
factor,  are therefore the surface must be nontrivial. The key difficulty in bounding the left side of \eqref{contra.ingred} 
in terms of the energy in a sector is that the cutoff function depends on $w/q$, so a derivative 
of this cutoff function produces a power of $1/x$. The resulting integral is controlled by 
using the specific algebraic form of the integrand on the left in \eqref{contra.ingred}. 
This somewhat remarkable fact is further evidence of the delicate nature of the blow-up procedure. 
 
\begin{proof}
First,  by translating and dilating, assume that $y_0=0$, and that the $Y_j$ and $Y_*$ are graphical over the 
vertical half-disc $\{x^2+y^2\le 1000, z = 0\}$, with graph functions $u_j$ and $u_*$, where $|\nabla u_j|, 
|\nabla u_*|\le 2\zeta<<1$. We can also assume that $\del_y u_j(0,0)=\alpha>0$ and $\partial_y u_*(0,0)=0$. 
Now, consider any subinterval  $[\beta_-,\beta_+] \subset [0,\alpha]$ with $\beta_+ \ll \alpha$. 
For any sequence $\beta_j\in [\beta_-,\beta_+]$, $\beta_j\to\beta_*$, the line $z=\beta_j  y$ intersects 
$\gamma_j = \del_\infty Y_j$ at a point $(y_j, u_j(0,y_j))$, $y_j>0$, and just as in \S \ref{uniform.isothermal},
we have $\lim_{j\to \infty}y_j=0$.  

Now dilate $Y_j$ by $\rho_j:=\frac{1}{y_j}$ to obtain a new surface $\tilde{Y}_j$ which converges to $Y'$, where 
$Y'$ is graphical over the entire vertical half-plane $\{z=0\}$ and passes through the fixed point $(1,\beta_*)$. 
If, for {\it any} such sequence $\beta_j$, $\calE(Y') \neq 0$, then the proof is complete. 

Otherwise, $\calE(Y') = 0$ so $Y'$ is totally geodesic and graphical over a half-plane, hence is the half-plane 
$\{z = \beta_* y, x > 0\}$. Rotating again to make this the $xy$-plane, the original graph function $u_j$ must satisfy 
$\del_y u_j(0,0) = 0$ while $\del_y u_*(0,0)\sim\alpha-\beta_*>\frac{\alpha}{2}>0$.  All of this is true for any
$\beta_* \in[\beta_-,\beta_+]$. Using Remark \ref{special.beta}, there exists a sequence $\beta_j$ such that 
$\int_{\tilde{Y}_j\cap \{1/4\le\sqrt{x^2+y^2+z^2\}}\le 4}|\Abar^j|^2d\bar{\mu}\to 0$.

By Lemma \ref{DeLell.appl}, there exists a sequence of hyperbolic isometries $\varphi_j$ such that $\varphi_j(\tilde{Y}_j)$ 
have all the properties listed there, and in particular  admit isothermal coordinates $(q_j,w_j)$ for which the conformal 
factor $\phi_j$ satisfies 
\begin{equation}
\label{known.bounds}
\begin{split}
||\nabla^{2}\phi_j||_{L^1(\varphi_n(\tilde{Y}_j))}+||\nabla\phi_j||_{L^2(\varphi_j(\tilde{Y}_j))}+||\phi_j||_{{\cal C}^0(\varphi_j(\tilde{Y}_j))} \\
\le \calE(\varphi_j(\tilde{Y}_j))+o(1)<2\e'(\zeta)+o(1).
\end{split} 
\end{equation}
Moreover, there is still a jump of $\alpha - \beta_*$ in the first derivative at the origin in these coordinates. 
The $\varphi_j(\tilde{Y}_j)$ are graphical over the disc $\{x^2+y^2\le 10, z=0\}$ (for simplicity, we denote 
the graph function by $u_j$) with $|\nabla u_j|\le 4\zeta$  and the image of the rectangle $0\le q_j\le1, |w_j|\le 1$
is entirely contained in ${\rm Graph} (u_j)$. Recall from Remark \ref{exists.extension} that the surfaces $\varphi_j(\tilde{Y}_j)$ 
admit an extension $Y^\flat_j$ which is a graph over the $xy$-plane with graph function $u_j$, where $u_j=0$ 
for $ \sqrt{x^2+y^2}\ge 50$. The bounds \eqref{known.bounds} continue to hold for this extended surface. 

Using Remark \ref{translate.bounds} and the smallness of the energy 
we derive that $|y|/ x \le 10$ and $x^2+y^2\le 10$ at all points in the sector 
\[
\calS_j:=\{ (q_j,w_j)\in Y_j, \ |w_j|/q_j \le 1, \ q^2_j+w^2_j\le 4\},
\]

We now claim that one of the following must be true:
\begin{enumerate}
\item[a)]  Either $\varphi_n(\tilde{Y}_j)$ converge to a {\it nontrivial} limit $\tilde{Y}_*$, or else
\item[b)]  there exists a sequence $\omega_j\rightarrow\infty$ such that the dilates $\omega_j\cdot \varphi_n(\tilde{Y}_j)$
converge to a  non-trivial limit $\tilde{Y}_*$.
\end{enumerate}
The theorem will be proved once we show that these are the only possibilities.

As many times before, write $\varphi_j(\tilde{Y}_j)$ as just $Y_j$. If alternative a) does not occur, then $Y_j$ converges 
to a vertical half-plane $Y_*$.  We claim that for some $\mu > 0$, there exists a sequence $P_j\in\calS_j$ such that 
either $|\nabla \phi_j|_g(P_{j}) \geq \mu$ or else $x\, |\Abar|_{\olg}(P_{j}) \geq \mu$. 

Observe that either of these two possibilities imply our theorem. Indeed, suppose the former of these is true and 
consider the dilated surfaces $\frac{1}{x(P_j)}\, Y_j$. The images 
$\tilde{P}_j$ of the points $P_j$ have height $x_j=1$ and $|y_j|\le 10$. Setting $\lambda_j:=\frac{1}{x(P_j)}$,
consider the isothermal coordinates 
\[
\tilde{q}_j(\lambda_j x, \lambda_j y) = \lambda_j q_j(x,y), \ \tilde{w}_j(\lambda_jx,\lambda_j y)=\lambda_j w_j(x,y),
\]
and the corresponding conformal factor $\tilde{\phi}_j(\lambda_jx,\lambda_jy)=\phi_j(x,y)$ on $\lambda_j Y_j$. 
Clearly, $|\nabla \tilde{\phi}_j|_{\olg}(\tilde{P}_j)\ge\mu$. Also, passing to a subsequence, $\tilde{P}_j\to \tilde{P}_*$
where $x(\tilde{P}_*)=1$, $|y(\tilde{P}_*)|\le 10$ and $|z(\tilde{P}_*)|\le 10\zeta$. 

Using the estimates in \cite[Thm.\ I.5]{riv1}, some subsequence of the surfaces $\lambda_j Y_j$ must converge, smoothly in the interior
and in $\calC^{0,\alpha}$ up to the boundary, to a surface $\tilde{Y}_*$, and this limit surface admits isothermal coordinates 
$(\tilde{q}_*,\tilde{w}_*)$ where $\tilde{q}_*=0$ on $\{x=0\}$ and $1/10\le |\tilde{q}|/|x| \le 10$.  
The convergence $(\tilde{q}_j,\tilde{w}_j,\tilde{\phi}_j)\to (\tilde{q}_*,\tilde{w}_*, \tilde{\phi})$ is smooth 
away from $\{x=0\}$. We claim that $\tilde{Y}_*$ can not be a vertical half-plane. Indeed, if it were, 
then following the same argument as in the second paragraph of \S 5, $\tilde{q}_* = Cx$ for some 
$1/10 \leq C \leq 10$, and in that case, the corresponding conformal factor 
$\tilde{\phi}_*$ would be constant. This contradicts the smooth convergence and the fact that 
$|\nabla \tilde{\phi}_j(\tilde{P}_j)|_{\olg}\ge\mu$. 


The proof that $x |\Abar|_{\olg}(P_j) \geq \mu$ implies the result is even simpler. Indeed, the same sequence
of dilations of $Y_j$ converges to a Willmore surface with $|\Abar|_{\olg}(\tilde{P}_*)\ge\mu$, and this must be nontrivial
since we know that it is graphical over the half-plane $\{z=0\}$ and hence cannot be a sphere. 

We have therefore reduced the proof to showing that conditions i) - vi) below lead to a contradiction. 
\begin{enumerate}
\item[i)] Each $Y_j$ is a graphical Willmore surface, with graph function $u_j$, over $\{x^2+y^2\le 10, x>0, z = 0\}$, 
with $\calE(Y_j) \le \e'(\zeta)$ and $\int_{Y_j}|\Abar_j|^2d\bar{\mu}\le M$ for some fixed $M<\infty$. 
The surface $Y_j$ extends to a (non-Willmore) graphical surface $Y^\flat_j$. The region 
$30\le \sqrt{x^2+y^2}$ is denoted $Y^\sharp_j$ and $u_j=0$ there. 
\item[ii)] Each $Y_j$, and its extension $Y^\flat_j$ too, admits an isothermal coordinate chart $(q_j,w_j)$ 
with conformal factor $\phi_j$ satisfying \eqref{known.bounds}.
\item[iii)] $Y_j \to Y_*:=\{x^2+y^2\le 10, z = 0\}$.
\item[iv)] The conformal factors $\phi_j$ satisfy $|\nabla\phi_j|_{g}\rightarrow 0$ uniformly 
in $\calS_j$. 
\item[v)] $x\cdot |\Abar_j|_{\olg}\to 0$ uniformly in $\calS_j$.
\item[vi)]  $|\int_0^1\partial_2e^{-\phi_j}\partial_qu_jdq_j|\ge\epsilon_0>0$. 

\end{enumerate}
 
To reach the contradiction it suffices to prove that conditions i) - v) contradict condition vi).
In other words, we need to show that i) -v) imply: 
\begin{equation}
\label{contra.ingred}
\lim_{j\to \infty} \left|\int_{(0,0)}^{(1,0)} \partial_2e^{-\phi_j}\partial_1u_j\, dq_j\right|=0.
\end{equation}

\medskip

\noindent {\it Proof of \eqref{contra.ingred}:} 
Recall that since the conformal factor $\phi_j$ is bounded, the quantities $|\partial_q|$ and $|\partial_x|$, 
and $|\partial_y|$, $|\partial_w|$ are comparable. In the following, $|\cdot|$ denotes the norm with respect 
to $dq^2 + dw^2$. In many expressions below, we suppress the subscripts $j$ for simplicity. 

The strategy is to express the second integral $\int_{(0,0)}^{(1,0)} \partial_2e^{-\phi_j}\partial_1u\, dq$ as the 
flux of the integral of a divergence over some part of the circular sector ${\cal S}_j$. Introduce 
polar coordinates $r_j=\sqrt{q^2_j+w^2_j}$ and $\theta_j$ with $\tan (\theta_j+\frac{\pi}{2})=\frac{w_j}{q_j}$,
so that ${\calS}_j:=\{0\le r_j\le 1, \pi/4\le \theta_j\le 3\pi/4\}\subset Y_j$. 
Let $\calS^l$ denote the region where $\pi/2\le \theta\le 3\pi/4$, and define $\chi^l(\theta)= (3-\frac{4\theta}{\pi})$ 
in ${\cal S}^l$.
 
By the divergence theorem,
\begin{equation}
\begin{split}
  &\int_{(0,0)}^{(1,0)} \partial_2e^{-\phi_j}\partial_1u\, dq= \int_{\calS^l}\Delta_{\olg} e^{-\phi_j}\partial_1 u\chi^l \, d\overline{\mu} +
\int_{\calS^l}\nabla_s e^{-\phi_j}\nabla^s(\partial_1u)\chi^l\, d\overline{\mu} \\
& +\frac{4}{\pi}\int_{\calS^l}e^{-2\phi_j}\del_{\theta}e^{-\phi_j}\frac{1}{r^2}\del_1u\, d\overline{\mu} + 
\int_{\pi/2}^{3\pi/4} (\partial_1u)\partial_1e^{-\phi_j}(1, \theta)\, d\theta.
\end{split}
\end{equation}
(The coefficient $\frac{4}{\pi}$ arises from $\del_\theta \chi^l$.) The final boundary term tends to zero since $Y_j$ 
converges to a vertical half-plane, so in particular $|\del u_j|, |\del \phi_j|\to 0$ away from $\{x=0\}$.

Now consider the bulk terms.  First, observe that the pointwise bounds on $\phi_j$ and on $|\nabla u|_{\olg}$ imply that 
$|\del_1 u|_{\olg}\le 3\zeta$. This uses the formula for the second fundamental form for a graph in $\RR^3$ and implies that:
\begin{equation}
\label{simple.control}
\int_{\calS^l}|\del^2 u|^2\, d\overline{\mu} \le 10 \int_{\calS^l}|\Abar|^2\, d\mu\le 10M.
\end{equation}
Using (\ref{key.eqn}) and $|\del_r u|_{\olg}\le 3\zeta$, we have
\begin{equation}
\label{old.terms1}
\begin{split}
& \left|\int_{\calS^l}\Delta_{\olg} e^{-\phi}\del_1u\chi^l\, d\bar{\mu} \right| \\
&\le 4\int_{\calS^l}e^{-\phi}\overline{H}^2|\del_1u|\, d\bar{\mu}+ \int_{\calS^l} |\Abring|^2e^{-\phi}|\del_1u|\, 
d\bar{\mu}+  \int_{\calS^l}|\nabla\phi|^2e^{-\phi}\, d\bar{\mu}. 
\end{split}
\end{equation}
Clearly,
\begin{equation} \label{easy}
\int_{\calS^l}\overline{H}^2|\partial_1u|\, d\bar{\mu}+4\int_{\calS^l}|\Abring|^2e^{-\phi}|\del_1u|\, d\bar{\mu}
\le 10 \int_{\calS^l}|\Abar|^2|\del_1u|\, d\bar{\mu}.
\end{equation}
In addition, using \eqref{known.bounds} and \eqref{simple.control},
\begin{equation}
\label{mixed}
\begin{split}
& \left|\int_{\calS^l}\nabla_se^{-\phi}\nabla^s(\partial_1u)\chi^l\, d\bar{\mu}\right|=
\left|\int_{\calS^l}e^{-2\phi}\del_se^{-\phi}\del^s(\del_1u)\chi^l\,  d\bar{\mu}\right| \\
&\le 4 \left(\int_{\calS^l} |\nabla\phi|^2\, d\bar{\mu}\right)^{\frac12} \left( \int_{\calS^l} |\del^2u|^2\, 
d\bar{\mu}\right)^{\frac12} \le 100\sqrt{M} \left(\int_{\calS^l} |\nabla\phi|^2\, d\bar{\mu}\right)^{\frac12}.
\end{split}
\end{equation}
The main issue is to control the term
\[
T_2:=\left|\int_{\calS^l}\frac{1}{r^2}e^{-2\phi}\del_{{\theta}}e^{-\phi}\del_1u\, d\bar{\mu}\right|. 
\]
Recall that by Lemma \ref{asympt.vertical}, $\del_1u=0$ on $\{q=0\}=\{x=0\}$, and also 
$\frac{|\del_1u|^2}{r^2}\le\frac{|\del_1u|^2}{q^2}$. The Hardy inequality now gives 
\begin{equation}
\label{hardy.app} 
\int_{\calS^l}\frac{|\partial_1 u|^2}{r^2}dqdw\le 10\int_{\cal B}\frac{|\partial_1 u|^2}{q^2}dqdw\le \int_D|\partial^2 u|^2dqdw\le 
100\int_{Y_j}|\Abar|^2_{\olg}d\overline{\mu}\le 100M. 
\end{equation}
where ${\cal B}:=\{0\le w\le 1, 0\le q\le 1\}$.  Thus 
\begin{equation}
\begin{split}
& T_2\le 10\int_{\calS^l} r^{-2} |\del_{\theta}\phi\del_1u|\, d\bar{\mu} \le 
10 \left( \int_{\calS^l}|\nabla \phi|_{\olg}^2\, d\bar{\mu} \right)^{\frac12} 
\left( \int_{\calS^l}  r^{-2} |\del_{1}u|^2\, d\bar{\mu}\right)^{\frac12} \\
&\le100 \left( \int_{\calS^l}|\nabla \phi|_{\olg}^2\, d\bar{\mu}\right)^{\frac12} 
\left( \int_{\cal B}|\del^2u|^2\, d\bar{\mu} \right)^{\frac12} \le 
100\sqrt{M} \left( \int_{\calS^l}|\nabla \phi|_{\olg}^2\, d\bar{\mu}\right)^{\frac12}. 
\end{split}
\end{equation}
 
We then claim that 
\begin{equation} \label{more.zero}
\lim_{j\to \infty} \int_{\calS^l}|\nabla\phi_j |_{\olg}^2\, d\bar{\mu}=0, 
\end{equation}
and 
\begin{equation}\label{2nd.zero}
\lim_{j\to \infty}\int_{\calS^l} |\Abar_j|^2|\del_1u_j|\, d\bar{\mu}=0.
\end{equation}  
These estimates will prove \eqref{contra.ingred}, and thus our theorem. 
\medskip

\noindent {\it Proof of \eqref{more.zero}:}  We assert first that on the family of lines 
$\ell_{\theta_0}:= \{0\le r\le 1, \theta=\theta_0\}$, $\frac{\pi}{2}\le\theta\le\frac{3\pi}{4}$,
there is a uniform bound 
\begin{equation}
\label{unif.bound}
\int_{\ell_{\theta_0}}|\del\phi_j| dr \le  \e'(\zeta) +M'.
\end{equation}
Before proving (\ref{unif.bound}), let us see how it proves  the estimate. 

Since $1/10\le q/x\le 10$ in $\calS^l$, we have
\[
r |\nabla\phi_j|_{\olg}\le 100 |\nabla\phi_j|_{g}
\]
in this sector, so that 
\[
\begin{split}
\int_{\calS^l}|\nabla\phi_j|_{\olg}^2\, d\bar{\mu}= & \int_{\pi/2}^{3\pi/4}\int_0^1|\del\phi_j||\del\phi_j|r\, drd\theta \\ 
& \le 100\sup_{\calS^l}|\nabla\phi_j|_{g} \sup_{\theta\in [\pi/2,3\pi/4]}\int_{\ell_{\theta}}|\del\phi_j|\, dr.
\end{split}
\]
Since the first factor tends to $0$ by the assumption v) above and the second one is bounded, we obtain \eqref{more.zero}.

\medskip

Thus matters are reduced to showing \eqref{unif.bound}. Recall from \eqref{known.bounds} that 
$||\del^2\phi_j||_{L^1(\RR^2)}\le \e'(\zeta)+o(1)$, where $\del$ is differentiation with respect to $(q_j,w_j)$.
Given any ray $\ell_{\theta_0}$, consider the right-angle rectangle $R_{\theta_0}\subset Y^\flat_j$ 
which is defined by four straight (with respect to the coordinates $q_j,w_j$) 
 line segments: $\ell_{\theta_0}$ is one segment, then $s^1,s^2$ are two line segments of length 50, emanating
  from the endpoints $(0,0)$ and $(\cos\theta_0, \sin\theta_0)$ 
 of $\ell_{\theta_0}$ and normal to it; finally $\ell'_{\theta_0}$  joins the other endpoints
  of $s^1,s^2$. Thus $\ell'_{\theta_0}$ is parallel to $\ell_{\theta_0}$ (with respect to the flat coordinates $q_j,w_j$) and lies in 
  the portion $Y^\sharp_j$ of $Y^\flat_j$ defined in Remark \ref{exists.extension}. 
   Let $n$ be the unit vector field (with respect to $dq^2 + dw^2$) normal to 
$\ell_{\theta_0}$, so $n$ is also normal to $\ell'_{\theta_0}$ and tangent to the lines $s^1, s^2$. 

Integrating $\del_n(\del\phi_j)$ over the rectangle $R_{\theta_0}$ and decomposing 
$\ell_{\theta_0}$ into sets where a given component of  $\partial\phi_j$ has 
constant sign, we obtain that 
\[
\int_{\ell_{\theta_0}}|\del\phi_j|\le \int_{R_{\theta_0}}|\del^2\phi_j|\, dqdw+ \int_{\ell'_{\theta_0}}|\del\phi_j|
\]
The first integral on the right is bounded above by \eqref{known.bounds}. We obtain a uniform upper 
bound on $\int_{\ell'_{\theta_0}}|\del\phi_j|$ by proving that $\del\phi_j$ is bounded above pointwise 
over $\ell'_{\theta_0}$, which is true because in $Y^\sharp_j$ we have $\Delta\phi_j=0$ and  $|\phi_j|$ is uniformly
bounded. This proves \eqref{unif.bound}.  \hfill $\Box$
  
\medskip

\noindent {\it Proof of \eqref{2nd.zero}:}  Recall that condition vi) implies that $\lim_{j\to\infty} \sup_{\calS^l} r|\Abar_j|=0$;
using Cauchy-Schwarz, the Hardy inequality and \eqref{simple.control}, we get
\begin{equation} \label{2nd.zeroPF}
\begin{split}
&\int_{\calS^l}|\Abar_j|^2|\partial_1u_j|\, d\bar{\mu}=
\int_{\calS^l}r|\Abar_j|\cdot |\Abar_j|\cdot \frac{1}{r}|\del_1 u_j|\, d\bar{\mu} \\
&\le \left( \sup_{\calS^l} r|\Abar_j|\right) \left( \int_{\calS^l} |\Abar_j|^2\, d\bar{\mu} \right)^{\frac12} 
\left( \int_{\calS^l} r^{-2}|\del_1 u_j|^2  \, d\bar{\mu} \right)^{\frac12} \\
& \le 10\left( \sup_{\calS^l} r|\Abar_j|\right) M\to 0.
\end{split}
\end{equation}    
This proves \eqref{contra.ingred}, and hence completes the proof of our theorem.
\end{proof}

\subsection{Finitely many bubbles}
\label{finitedistance}
We conclude by showing that there can exist at most $N$ nontrivial nonisometric limits in the sequence $Y_j$,
which completes the proof of Theorem \ref{jump.implies.bubble}. In other words, given $Y_j$ and any sequence of
isometries $\varphi_j$ such that $\varphi_j(Y_j)$ converges to a limiting Willmore surface $Y_*$ with $\calE(Y_*)> 0$,
there can be at most $N$ distinct possible limits $Y_*$. 

Arguing as usual by contradiction, assume there exist $N+1$ non-isometric limits, $Y_{*,1},\dots, Y_{*,N+1}$, 
with induced metrics  $g_{*,k}$, $k\leq N+1$. 
Since the limits $Y_{*,k}$ are non-trivial, there exist points $A_k$ on each $Y_{*,k}$ such that
the intrinsic balls $B^1_{g_{*,k}}(A_k)$ have non-zero energy. 
   
The fact that the $Y_{*,k}$ arise as limits of $\varphi_j(Y_j)$ (with $\calC^\infty$ convergence on compact sets) 
gives balls $B^1_{g_j}(C_{j,k})\subset Y_j$ with $\varphi_k(B^1_{g_j}(C_{j,k}))\to B^1_{g_{*,k}}(A_k)$. In particular we 
may assume that there is an $\epsilon>0$ such that $\calE[B^1_{g_j}(C_{j,k})]>\epsilon$ for all $j,k$.     
We then claim that for any pair of distinct values $k\ne l$ the points 
 $C_{j,k},C_{j,l}\in Y_j$ drift infinitely far apart, i.e.\ 
\begin{equation}
\label{the.claim}
{\rm dist}_{g_j}(C_{j,k},C_{j,l})\to \infty.
\end{equation} 
If we can prove this, then using the triangle inquality, there is a sequence of points $C_{j,k(j)}\in Y_j$
with $\min_{r}{\rm dist}_{g_j}(O_r, C_{j,k(j)})\to \infty$. But this cannot hold since then 
\begin{equation}
\label{lower.bd}
\begin{split}
&{\cal E}_p(Y_j)\ge \int_{B^1_{g_j}(C_{j,k})}[\min_{r}{\rm dist}_{g_j} (O_r, C_{j,k(j)})]^{2p}|A_j|^2\, d\mu_j  \\
&\ge [\min_{r}{\rm dist}_{g_j}(O_r, C_{j,k(j)})]^{2p}\epsilon^2\to\infty.
\end{split}   
 \end{equation}

We have reduced to proving \eqref{the.claim}. But if this were not true, then 
a large enough ball centered at $C_{j,k}$ must contain $C_{j,l}$, which would imply that the limit surfaces 
$Y_{*,k}$ and $Y_{*,l}$ coincide up to a hyperbolic isometry; this contradicts our assumption.

\section{Examples}
In this final section we show that the putative modes of convergence described above actually occur. 
Namely, we exhibit sequences $Y_j$ of complete, properly embedded minimal (and therefore Willmore) surfaces in $\HH^3$ 
with fixed genus which lose energy in the limit because some portions separate and disappear toward
infinity. These $Y_j$ have energy tending to zero and converge smoothly away from a finite number 
of points on the boundary curve at infinity. The limit is another complete, properly embedded surface $Y_*$, 
and we find such sequences where the genus of $Y_*$ is strictly less than that of each of the $Y_j$. 
In other words, there can be a loss of genus in the limit. The construction of these surfaces
proceeds by a fairly standard gluing result. There are many very similar ways to prove such theorems, and
we follow a method used in the papers \cite{MaS,MaMaR,MP}. Since this method is well documented in these papers,
we provide only a sketch of the argument. 
\begin{theorem}
Choose a finite number, $Y_1, \ldots, Y_k$, of complete, properly embedded minimal surfaces, each with finite energy. 
Suppose that each $\gamma_r= \del_\infty Y_r$, $r=1,\dots, k$ is a $\calC^2$ curve, and assume also that each $Y_r, r=1,\dots, k$ 
is nondegenerate in the sense that it admits no Jacobi fields which decay at $\gamma_r$. Then there is a family of complete, 
properly embedded minimal surfaces $Y_t$ with boundary curves $\del_\infty Y_t = \gamma_t$ a small perturbation of the 
unit circle. These boundary curves converge in $\calC^2$ to the unit circle away from $k$ distinct points $q_1, \ldots, q_k$. 
Furthermore, there exist rescalings of $Y_t$ at $q_j$ which converge to an isometric copy of $Y_r$. Finally, 
\[
\calE(Y_t) = \sum_{r=1}^k \calE(Y_r) + o(1)
\]
as $t \to \infty$.
\end{theorem}
\begin{proof}  There are three steps: we first construct a family of approximate solutions $Y_t'$ which 
are approximately minimal and have the stated concentration properties; we next analyze the mapping
properties of the Jacobi operators on these surfaces, focusing on estimates which are uniform in
the parameter $t$; the final step is to perturb $Y_t'$ to a minimal surface $Y_t$ when $t$ is sufficiently
large.

\medskip

\noindent {\bf Approximate solutions:}  First, choose two separate collections of points $p_1, \ldots, p_k$ 
and $q_1, \ldots, q_k$ on the unit circle $S^1$ in the boundary at infinity $\{x=0\}$, such that no two of 
these points coincide. For simplicity of notation, assume that $p_r = -q_r$ below.  Next, fix points $p_r', q_r' \in 
\gamma_r$, $r = 1, \ldots, k$, and choose a hyperbolic isometry $F_r$ which carries $p_r'$ to $p_r$ and $q_r'$ to $q_r$, and set 
$Y_r' = F_r(Y_r)$. Finally, let $M_{r,t}$ be the family of hyperbolic dilations with source $p_r$ and sink $q_r$,
and set $Y_{r,t} = M_{r,t}(Y_r')$. 

As $t \to +\infty$, the surfaces $Y_{r,t}$ converge locally uniformly in $\calC^2$ in the region $\overline{\HH^3} 
\setminus \{q_r\}$ to the totally geodesic hemisphere $H$ bounded by the unit circle, and this convergence is
$\calC^\infty$ away from $\{x=0\}$. In particular, $\gamma_{r,t} := \del_\infty Y_{r,t}$ converges in $\calC^2$ 
away from the point $q_r$. Applying the inverse dilations $M_{r,-t}$, we see that rescalings of $Y_t'$ converge to $Y_r'$, 
which is an isometric copy of $Y_r$.

For each $r$, choose a closed spherical cap $A_r$ (intersected with the half-space $x \geq 0$) centered at $q_r$ 
in the unit hemisphere $H$. We can do this so that these caps are disjoint from one another, and we then
let $B_r = H \setminus A_r$. Choose a slightly larger spherical cap $B_r' \supset B_r$, so $B_r' \cap A_r$ is diffeomorphic 
to a rectangle. Let $A_r'$ be the complement of $B_r'$ in $H$. 
By the convergence explained in the last paragraph, some portion $B_{r,t}' \subset Y_{r,t}$ 
is a normal graph over $B_r'$ with graph function $u_{r,t}$ converging to $0$ in $\calC^2(B_r') \cap \calC^\infty( B_r' 
\setminus (B_r' \cap \{x=0\})$. Finally, choose a smooth nonnegative cutoff function $\chi_r$ which
has support in $A_r \setminus (A_r \cap B_r')$ and which equals $1$ in $A_r'$. Let $Y_{r,t}'$ be the surface 
which agrees with $Y_{r,t}$ over $A_r'$ and which has graph function $\chi_r u_{r,t}$ over $B_r'$.  

By construction, each $Y_{r,t}'$ coincides with the totally geodesic hemisphere in the region $B_r$,
and this region is disjoint from all of the other regions $A_i$, $i \neq r$. This means that we may define 
the surface $Y_t'$ to be the superposition of these $k$ separate surfaces, since they all agree
on the complement of the union of the $A_r$ in the hemisphere $H$.  

Observe that these surfaces are minimal in $H \setminus (A_1 \cup \ldots \cup A_k)$ and
in $A_1' \cup \ldots \cup A_k'$, and the discrepancy from being minimal in the overlap regions
tends to $0$ as $t \to \infty$. 

\medskip

\noindent {\bf Analysis of the Jacobi operator}
Consider the Jacobi operator 
\[
L_r = \Delta_{Y_r} + |A_r|^2 - 2
\]
on the surface $Y_r$. This operator has continuous spectrum filling out the half-line $(-\infty, -9/4]$ and a finite number of
$L^2$ eigenvalues above that ray. The assumption that $Y_r$ is nondegenerate means that $L_r: H^2(Y_r) \to L^2(Y_r)$
is an isomorphism, i.e.\ $0$ is not an $L^2$ eigenvalue. It is also the case, cf.\ \cite{AM}, that under this condition, 
$L_r$ is an isomorphism on other function spaces better suited for the gluing argument. In particular, 
let $x^\delta \calC^{k,\alpha}$ denote the intrinsic H\"older space (relative to the metric on $Y_r$ induced from the 
hyperbolic metric) weighted by the function $x^\delta$, where $x$ is the upper half-space coordinate restricted to $Y_r$. 
As described carefully in \cite{AM}, if $0 < \delta < 3$, then
\[
L_r:  x^\delta \calC^{2,\delta}(Y_r) \longrightarrow x^\delta \calC^{0,\alpha}(Y_r)
\]
is an isomorphism. Denote its inverse by $G_r$. It is very important that we do not just know the existence of
this operator abstractly, but realize that it is a pseudodifferential operator for which we have a rather explicit description
of the asymptotic behaviour of its Schwartz kernel. 

Let us now define a family of weighted H\"older spaces on the surfaces $Y'_t$.  We have already defined
the cutoff functions $\chi_r$, $r = 1, \ldots, k$, and it is clearly possible to add one extra smooth nonnegative 
function $\chi_0$ which equals $1$ on $H\setminus (A_1 \cup \ldots \cup A_k)$ and is supported
away from $A_1' \cup \ldots, A_k'$, such that $\{\chi_0, \ldots, \chi_k\}$ is a partition of unity on $Y_t'$. 
(We suppress the dependence on $t$ in the $\chi_r$.) Now define
\begin{multline*}
\calC^{\ell,\alpha}_{\delta, t} (Y_t') = \{u = \sum_{r=0}^k \chi_r u_r, \ \mbox{where} \ u_r = (M_t \circ F_r)^* v_r,\ \ v_r \in 
x^\delta \calC^{\ell,\delta}(Y_r),\ r = 1, \ldots, k,\\ \mbox{and}\ u_0 \equiv v_0 \in x^\delta \calC^{\ell,\alpha}(H)\}, 
\end{multline*}
endowed with the norm
\[
||u||_{\delta,t} = \sum_{r=0}^k ||v_r||_{\ell,\alpha,\delta}.
\]
Notice that the elements of $\calC^{\ell,\alpha}_{\delta, t}(Y_t')$ coincide with those in $x^\delta \calC^{\ell,\alpha}(Y_t')$, 
but the norm 
in which there is a hidden extra $t$ dependence, so in particular this norm is not uniformly equivalent as $t \nearrow \infty$
to the standard norm on $x^\delta \calC^{\ell,\alpha}(Y_t')$, which is given by an expression similar to the one above,
but using the summands $u_r$ instead of $v_r$. 

Next, we can transfer the inverse $G_r$ on $Y_r$ using the mapping $M_t \circ F_r$ to
an operator $G_{r,t}$ on $Y_{r,t}'$, and then define
\[
\widetilde{G}_t = \sum_{r=0}^k \tilde{\chi}_r G_{r,t} \chi_r.
\]
Here each $\tilde{\chi}_r$ is a nonnegative smooth cutoff function which is equal to $1$ on the support
of $\chi_r$ and vanishes outside a larger neighbourhood.  We compute that if $L_t$ denotes the
Jacobi operator on $Y_t'$, then
\[
L_t \widetilde{G}_t = \mbox{Id} - \sum_{r=0}^k [ L_t \tilde{\chi}_r] G_{r,t} \chi_r := \mbox{Id} - K_t.
\]
The operator $K_t$ is a smoothing operator; this is because the supports of $[L_t ,\tilde{\chi}_r]$ and $\chi_r$
are disjoint from one another, and because $G_r$ is a pseudodifferential operator, the Schwartz kernel
of which is necessarily singular only along the diagonal. Moreover, it is possible to choose the supports
of these two functions, $[L_t ,\tilde{\chi}_r]$ and $\chi_r$, very far from one another. On the other hand,
the Schwartz kernel of $G_{r,t}$ has a decay profile equivalent to the one of $G_r$; namely, 
$G_r(z,z') \leq C \exp (-3 \, d_{Y_r}(z,z'))$. Taking these facts together, and arguing exactly as
in \cite{MS}, we conclude that the norm of $K_t$ as a mapping on $\calC^{\ell,\alpha}_{\delta, t}$ for
any fixed $\delta \in (0,3)$, can be made as small as desired, uniformly in $t$, by choosing the supports 
of these cutoff functions appropriately.  We conclude from this that
\[
L_t: \calC^{2,\alpha}_{\delta,t}(Y_t') \longrightarrow \calC^{0,\alpha}_{\delta, t}
\]
is an isomorphism for all $t > 0$ whenever $0 < \delta < 3$, and the norm of its inverse is uniformly 
bounded in $t$ as $t \to \infty$. 

\medskip

\noindent{\bf The  gluing construction}
If $\nu$ is the (hyperbolic) unit normal to $Y_t'$ and $\phi$ is any function on $Y_t'$, then define
the normal graph
\[
Y_{t,\phi} = \{ \exp_z( \phi(z) \nu(z)): z \in Y_t'\}.
\]
Let $\calM$ denote the minimal surface operators on $Y_t'$, i.e.\ $\calM(\phi)$ is the (hyperbolic)  mean
curvature function of $Y_{t,\phi}$, viewed as a graph over $Y_t'$. This is a second order quasilinear
operator which can be written as a small perturbation of the minimal surface operators for normal
graphs on $Y_{r,t}$ and $H$, but the main thing we need to know about it is that its linearization
at $\phi = 0$ is simply the Jacobi operator $L_t$. 

The perturbation argument is standard. Set $\calM(0) = f$. It is not hard to see that
$||f||_{0,\alpha,\delta} \to 0$ as $t \to \infty$. Expand $\calM(\phi) = 0$ as
\[
f + L_t\phi + Q_t(\phi) = 0 \Longrightarrow  L_t \phi = - f - Q_t(\phi);
\]
here $Q_t$ is quadratic remainder term involving the terms $\phi$, $\nabla \phi$
and $\nabla^2 \phi$ which satisfies
\[
||Q_t(\phi)||_{0,\alpha,\delta} \leq C ||\phi||^2_{2,\alpha,\delta}
\]
and
\[
||Q_t(\phi) - Q_t(\psi) ||_{0,\alpha,\delta} \leq C (||\phi||_{2,\alpha,\delta} + ||\psi||_{2,\alpha,\delta}) ||\phi-\psi||_{2,\alpha,\delta}. 
\]
The equation 
\[
\phi = - G_t (f + Q_t(\phi)),
\]
can then be solved using the estimates above by a straightforward contraction mapping argument. 

\medskip

It is easy from the construction to see that if $t$ is quite large, then $||\phi||_{2,\alpha,\delta}$ is
small and the surface $Y_t := Y_{t,\phi}$ is embedded. Since $\phi \to 0$ at $\del_\infty Y_t'$, we see
that the new surface has the same boundary curve at infinity.  The fact that $Y_t$ converges
in $\calC^2$ away from the points $q_1, \ldots, q_k$ follows directly from the construction. 
\end{proof}

\end{document}